\documentclass[10pt]{amsart}
\usepackage[T1]{fontenc}
\usepackage{geometry}
\usepackage[latin1] {inputenc}
\usepackage{amsmath}
\usepackage{amsfonts, amssymb, textcomp}
\usepackage[colorlinks=flase, linkcolor=red,urlcolor=green, citecolor=blue]{hyperref}
\usepackage{subeqnarray}

\usepackage{latexsym}
\usepackage{fancyhdr}
\usepackage{longtable}
\usepackage{amsmath, amssymb}
\usepackage{graphicx}
\usepackage{enumerate}

\numberwithin{equation}{section}

\theoremstyle{plain}
\newtheorem{proposition}{Proposition}[section]
\newtheorem{theorem}[proposition]{Theorem}
\newtheorem{lemma}[proposition]{Lemma}
\newtheorem{corollary}[proposition]{Corollary}
\newtheorem{definition}[proposition]{Definition}
\newtheorem{example}[proposition]{Example}

\newtheorem{remark}[proposition]{Remark}

\newcommand{\RR}{\mathbb{R}}
\newcommand{\CC}{\mathbb{C}}
\newcommand{\NN}{\mathbb{N}}

\let\on=\operatorname

\newenvironment{demo}[1]{\par\smallskip\noindent{\bf #1.}}{\par\smallskip}

\makeatletter
\@namedef{subjclassname@2020}{%
	\textup{2020} Mathematics Subject Classification}
\makeatother

\newsavebox{\fmbox}
\newenvironment{fmpage}[1]
{\begin{lrbox}{\fmbox}\begin{minipage}{#1}}
		{\end{minipage}\end{lrbox}\fbox{\usebox{\fmbox}}}

\title{On generalized definitions of ultradifferentiable classes}

\author[J.~Jim\'{e}nez-Garrido, D.N.~Nenning, and G.~Schindl]{Javier Jim\'{e}nez-Garrido, David Nicolas Nenning and Gerhard Schindl}

\address{J.~Jim\'{e}nez-Garrido: Departamento de Matem\'aticas, Estad\'istica y Computaci\'on,
	Universidad de Cantabria, Avda. de los Castros, s/n, 39005 Santander, Spain.
	 Instituto de Investigaci\'on en Matem\'aticas IMUVA, Universidad de Valladolid, Spain.}
\email{jesusjavier.jimenez@unican.es}

\address{D. N.~Nenning: Fakult\"at f\"ur Mathematik, Universit\"at Wien, Oskar-Morgenstern-Platz~1, A-1090 Wien, Austria.}
\email{david.nicolas.nenning@univie.ac.at}

\address{G.~Schindl: Fakult\"at f\"ur Mathematik, Universit\"at Wien, Oskar-Morgenstern-Platz~1, A-1090 Wien, Austria.}
\email{gerhard.schindl@univie.ac.at}

\begin{document}
	
	\begin{abstract}
		We show that the ultradifferentiable-like classes of smooth functions introduced and studied by S. Pilipovi\'{c}, N. Teofanov and F. Tomi\'{c} are special cases of the general framework of spaces of ultradifferentiable functions defined in terms of weight matrices in the sense of A. Rainer and the third author. We study classes ``beyond geometric growth factors'' defined in terms of a weight sequence and an exponent sequence, prove that these new types admit a weight matrix representation and transfer known results from the matrix-type to such a non-standard ultradifferentiable setting.
	\end{abstract}
	
	\thanks{J. Jim\'{e}nez-Garrido is supported by project PID2019-105621GB-I00, D.N. Nenning and G. Schindl are supported by FWF-Project P33417-N}
	\keywords{Classes of ultradifferentiable functions, weight sequences and weight matrices}
	\subjclass[2020]{26A12, 26A48, 46A13, 46E10}
	\date{\today}
	
	\maketitle
	
	\section{Introduction}\label{Introduction}
	Spaces of ultradifferentiable functions are sub-classes of smooth functions with certain restrictions on the growth of their derivatives. Two classical approaches are commonly considered, either the restrictions are expressed by means of a weight sequence $M=(M_j)_j$, also called {\itshape Denjoy-Carleman classes} (e.g. see \cite{Komatsu73}), or by means of a weight function $\omega$ also called {\itshape Braun-Meise-Taylor classes}; see \cite{BraunMeiseTaylor90}. More precisely (in the one-dimensional case) for each compact set $K$, the sets
	\begin{equation}\label{introequ}
		\left\{\frac{f^{(j)}(x)}{h^jM_j}\,\,:\,\, j\in\NN,\,\, x\in K \right\}, \quad \text{respectively}\quad \left\{\frac{f^{(j)}(x)}{\exp(\frac{1}{h}\varphi^{*}_\omega(hj))}\,\,:\,\, j\in\NN,\,\, x\in K \right\},
	\end{equation}
	are required to be bounded, where $\varphi^{*}_\omega$ denotes the Young-conjugate of $t\mapsto \omega(e^t)$. We shall mention that in the second situation the classes can be defined directly by using $\omega$ and controlling the decay of the Fourier transform $\widehat{f}$ with growth factors $t\mapsto\exp(h\omega(t))$, $h>0$. In fact, this is the original description; see \cite{Bjorck66} and also the discussion in \cite{BraunMeiseTaylor90} where the original approach is transferred to the boundedness condition expressed in \eqref{introequ}.
	
	In the literature standard growth and regularity conditions are assumed for $M$ and $\omega$ and in both settings we can consider two different types of spaces: For the {\itshape Roumieu-type} the boundedness of the sets in \eqref{introequ} is required for some $h>0$, whereas for the {\itshape Beurling-type} it is required for all $h>0$.
	
	The most well-known examples are the {\itshape Gevrey sequences} of type $\alpha>0$ with $G^{\alpha}_j:= j^{\alpha j}$ for $j\in \NN$ (or equivalently use $M^{\alpha}_j:=j!^{\alpha}$). Alternatively, one can use the function $t\mapsto t^{1/\alpha}=:\omega_{\alpha}(t)$.\vspace{6pt}
	
	It is then a natural question how both classical settings are related. In \cite{BonetMeiseMelikhov07} this problem is studied and it has been shown that in general both approaches are mutually distinct. However, based on this work, in \cite{dissertation} and \cite{compositionpaper} A. Rainer and the third author have introduced the notion of {\itshape weight matrices} $\mathcal{M}=\{M^{(x)}: x>0\}$ which allows to treat both classical methods in a unified way and to transfer proofs from one context to the other. This can be achieved when considering $\mathcal{M}=\{M\}$ for the weight sequence and the so-called {\itshape associated weight matrix} $\mathcal{W}:=\{W^{(\ell)}: \ell>0\}$ with $W^{(\ell)}_j:=\exp(\frac{1}{\ell}\varphi^{*}_\omega(\ell j))$ in the weight function case. But one is also able to describe more classes, e.g. take the Gevrey matrix $\mathcal{G}:=\{G^{\alpha}: \alpha>1\}$; see \cite[Thm. 5.22]{compositionpaper}.\vspace{6pt}

	A second recent generalization was presented by S. Pilipovi\'{c}, N. Teofanov and F. Tomi\'{c}; see \cite{PTT2015}. For given parameters $\tau>0$ and $\sigma>1$ they consider the sequence $M^{\tau,\sigma}_j:= j^{\tau j^\sigma}$. However, in their definition the geometric growth factor $h^j$ appearing in \eqref{introequ} is replaced by $h^{j^{\sigma}}$. Observe that the growth of $j\mapsto h^{j^{\sigma}}$  is closely connected with $j\mapsto M^{\tau,\sigma}_j$. The authors called their framework ``beyond Gevrey regularity''  because  $M^{\tau,1}_j= j^{\tau j}$ for $\sigma=1$, i.e., the Gevrey sequence of type $\tau>0$.  Since all the classes considered in this work are, in some sense, generalizations of Gevrey classes, these spaces will be called {\itshape Pilipovi\'{c}-Teofanov-Tomi\'{c} classes}, or PTT-classes for short. \vspace{6pt}
	
	
	The difference between the growth of $j\mapsto h^j$ and $j\mapsto h^{j^{\sigma}}$ suggests that the PTT-classes can be viewed as ``non-standard ultradifferentiable classes'' and one can ask how both generalizations are related. In the introduction of \cite{PTT2021} it was claimed that the PTT-classes are not covered by the weight matrix approach which is due to the different growth of the factors mentioned before. However, the aim of this paper is to show that also the PTT-classes are contained in the weight matrix approach. \vspace{6pt}
	
	In fact, we treat a more abstract setting by considering an {\itshape exponent sequence} $\Phi=(\Phi_j)_{j\in\NN}$ and by replacing in \eqref{introequ} the growth $j\mapsto h^j$ by $j\mapsto h^{\Phi_j}$. This notion yields ``ultradifferentiable classes beyond geometric growth factors'' and we show that under mild regularity and growth assumptions on $\Phi$ such spaces admit a representation as weight matrix classes (as locally convex vector spaces) by involving the canonical matrix
	\[
	\mathcal{M}_{M,\Phi}:=\{M^{(c,\Phi)}: c>0\}, \quad M^{(c,\Phi)}_j:=c^{\Phi_j}M_j.
	\]
    Applying this main result to the PTT-classes, we are also able to see that when both $\sigma>1$ and $\tau>0$ are fixed then the corresponding space cannot be represented by a single weight sequence $M$ or by a weight function $\omega$; i.e., one requires the general weight matrix setting to describe these classes. In other words PTT-classes constitute genuine examples of ultradifferentiable classes defined by weight matrices.\vspace{6pt}
	
	On the other hand, in the very recent paper \cite{TT2022} it is shown that when only $\sigma>1$ is fixed and when one considers matrix-type classes with parameter $\tau>0$, i.e. PTT-limit classes, then these spaces can alternatively be defined in terms of a weight function (in particular of a so-called associated weight function). We give an independent proof of this result by applying purely weight matrix techniques; see Theorem \ref{varyingtaucharact}.\vspace{6pt}
	
	The paper is structured as follows: In Section \ref{weightcondsection} all necessary and relevant conditions on weight sequences, weight functions and weight matrices are given and the corresponding classes are defined. In Section \ref{ultrabeyondsection} we introduce ultradifferentiable spaces ``beyond geometric growth factors'' and prove in Section \ref{ultrabeyondvsweightmatrixsection} the main characterization results, i.e., Theorems \ref{PPTasweightmatrixthm}, \ref{Phinecprop1} and \ref{Phinecprop}, showing that, in particular, the PTT-classes can be represented as weight matrix spaces. In Section \ref{PTTsection} we apply this fact for fixed parameters $\tau>0$, $\sigma>1$, and study properties of the relevant matrix $\mathcal{M}^{\tau,\sigma}$ in order to transfer known results from the matrix setting to PTT-classes. In Section \ref{PTTnonfixedparametersect} this is done analogously for so-called limit classes when fixing $\sigma$ but letting $\tau\rightarrow 0$ resp. $\tau\rightarrow+\infty$. It is shown that such spaces can be represented as Braun-Meise-Taylor classes (see Theorem \ref{varyingtaucharact}) and  satisfy additional properties since in this weight structure both {\itshape mixed moderate growth conditions} of the particular type are valid.
	
	\subsection*{Acknowledgements}
	We wish to thank J. Vindas for pointing out additional results available for matrix classes; more precisely for bringing, what is now property (f) in Sections \ref{sec:pttresults} and \ref{sec:pttlimresults}, to our attention. In addition he suggested to consider \cite{debrouwkalmes2022}, whose implications are the content of Section \ref{sec:possibleresult}.\par
	And we thank N. Teofanov and F. Tomi\'c for forwarding their preprint of \cite{TT2022} and the subsequent helpful discussions.

	\section{Weights and conditions}\label{weightcondsection}
	\subsection{General notation}
	We write $\NN:=\{0,1,2,\dots\}$ and $\NN_{>0}:=\{1,2,3,\dots\}$. With $\mathcal{E}$ we denote the class of all smooth functions. We use the standard multi-index notation and $f^{(\alpha)}$, $\alpha\in\NN^d$, stands for the $\alpha$-th derivative of a given smooth function $f$ (defined in $\RR^d$).
	
	Occasionally, we write the symbol $[\cdot]$ if we mean either $\{\cdot\}$ (Roumieu-type) or $(\cdot)$ (Beurling-type) for spaces and growth conditions.
	
	\subsection{Weight sequences}\label{weightsequences}
	Given a sequence $M=(M_j)_j\in\RR_{>0}^{\NN}$ we also use $m=(m_j)_j$ defined by $m_j:=\frac{M_j}{j!}$ and $\mu_j:=\frac{M_j}{M_{j-1}}$, $j\ge 1$, and set $\mu_0:=1$. Analogously these conventions are used for all other appearing sequences, i.e., $N \leftrightarrow n \leftrightarrow \nu$ etc. $M$ is called {\itshape normalized} if $1=M_0\le M_1$ holds true.
	
	$M$ is called {\itshape log-convex}, denoted by \hypertarget{lc}{$(\text{lc})$}, if
	$$\forall\;j\in\NN_{>0}:\;M_j^2\le M_{j-1} M_{j+1},$$
	equivalently if $\mu$ is nondecreasing. If $M$ is log-convex and normalized, then both $M$ and $j\mapsto(M_j)^{1/j}$ are nondecreasing. In this case we get $M_j\ge 1$ for all $j\ge 0$ and
	\begin{equation*}\label{mucompare1}
		\forall\;j\in\NN_{>0}:\;\;\;(M_j)^{1/j}\le\mu_j.
	\end{equation*}
	Moreover we get $M_jM_k\le M_{j+k}$ for all $j,k\in\NN$; e.g. see \cite[Lemma 2.0.4, Lemma 2.0.6]{diploma}.
	
	If $m$ is log-convex, then $M$ is also log-convex and in this case we call $M$ {\itshape strongly log-convex} and write that $M$ is \hypertarget{slc}{$(\text{slc})$}.
	
	For any $M=(M_j)_j\in\RR_{>0}^{\NN}$ it is well-known that
	\begin{equation}\label{mucompare}
		\liminf_{j\rightarrow\infty}\mu_j\le\liminf_{j\rightarrow\infty}(M_j)^{1/j}\le\limsup_{j\rightarrow\infty}(M_j)^{1/j}\le\limsup_{j\rightarrow\infty}\mu_j.
	\end{equation}
	
	For convenience we introduce the following set of sequences:
	$$\hypertarget{LCset}{\mathcal{LC}}:=\{M\in\RR_{>0}^{\NN}:\;M\;\text{is normalized, log-convex},\;\lim_{j\rightarrow\infty}(M_j)^{1/j}=\infty\}.$$
	
	$M$ has {\itshape moderate growth}, denoted by \hypertarget{mg}{$(\text{mg})$}, if
	$$\exists\;C\ge 1\;\forall\;j,k\in\NN:\;M_{j+k}\le C^{j+k+1} M_j M_k.$$
	A weaker condition is {\itshape derivation closedness}, denoted by \hypertarget{dc}{$(\text{dc})$}, if
	$$\exists\;A\ge 1\;\forall\;j\in\NN:\;M_{j+1}\le A^{j+1} M_j\Leftrightarrow\mu_{j+1}\le A^{j+1}.$$
	
	$M$ is called {\itshape non-quasianalytic}, denoted by \hypertarget{nq}{$(\text{nq})$}, if
	$$\sum_{j\ge 1}\frac{1}{\mu_j}<+\infty.$$
	In the literature \hyperlink{mg}{$(\on{mg})$} is also known under {\itshape stability of ultradifferential operators} or $(M.2)$,  \hyperlink{dc}{$(\on{dc})$} under $(M.2)'$ and \hyperlink{nq}{$(\on{nq})$} under $(M.3)'$; see \cite{Komatsu73}. It is also known that for log-convex (normalized) weight sequences \hyperlink{nq}{$(\on{nq})$} is equivalent to
	$$\sum_{j\ge 1}\frac{1}{(M_j)^{1/j}}<+\infty,$$
	which holds by the so-called {\itshape Carleman-inequality}; see \cite[Prop. 4.1.7]{diploma} and the references therein.\vspace{6pt}
	
	$M$ has \hypertarget{beta1}{$(\beta_1)$} (named after \cite{petzsche}) if
	$$\exists\;Q\in\NN_{>0}:\;\liminf_{j\rightarrow\infty}\frac{\mu_{Qj}}{\mu_j}>Q,$$
	and \hypertarget{gamma1}{$(\gamma_1)$} if
	$$\sup_{j\in\NN_{>0}}\frac{\mu_j}{j}\sum_{k\ge j}\frac{1}{\mu_k}<\infty.$$
	In \cite[Proposition 1.1]{petzsche} it has been shown that for $M\in\hyperlink{LCset}{\mathcal{LC}}$ both conditions are equivalent and in the literature \hyperlink{gamma1}{$(\gamma_1)$} is also called ``strong nonquasianalyticity condition''. In \cite{Komatsu73} it is denoted by $(M.3)$ (in fact, there $\frac{\mu_j}{j}$ is replaced by $\frac{\mu_j}{j-1}$ for $j\ge 2$ but which is equivalent to having \hyperlink{gamma1}{$(\gamma_1)$}).
	
	A weaker condition on $M$ is \hypertarget{beta3}{$(\beta_3)$} (named after \cite{dissertation}, see also \cite{BonetMeiseMelikhov07}) which reads as follows:
	$$\exists\;Q\in\NN_{>0}:\;\liminf_{j\rightarrow\infty}\frac{\mu_{Qj}}{\mu_j}>1.$$
	
	For two weight sequences $M=(M_j)_j$ and $N=(N_j)_j$ we write $M\hypertarget{mpreceq}{\preceq}N$ if
	$$\sup_{j\in\NN_{>0}}\left(\frac{M_j}{N_j}\right)^{1/j}<\infty,$$
	and call them equivalent, denoted by $M\hypertarget{approx}{\approx}N$, if
	$$M\hyperlink{mpreceq}{\preceq}N\;\text{and}\;N\hyperlink{mpreceq}{\preceq}M.$$
	In the relations above one can replace $M$ and $N$ simultaneously by $m$ and $n$ because $M\hyperlink{mpreceq}{\preceq}N\Leftrightarrow m\hyperlink{mpreceq}{\preceq}n$. Let us also write $M\le N$ if $M_j\le N_j$ for all $j\in\NN$. Finally, we write $M\hypertarget{triangle}{\vartriangleleft}N$, if
	$$\lim_{j\rightarrow\infty}\left(\frac{M_j}{N_j}\right)^{1/j}=0.$$
	For any $s\ge 0$ we set $G^s:=(j!^s)_{j\in\NN}$, so for $s>0$ this denotes the classical {\itshape Gevrey sequence} of index/order $s$.

	\subsection{Weight functions}
	According to \cite[Sect. 2.1]{optimalRoumieu} and \cite[Sect. 2.2]{optimalBeurling} a function $\omega:[0,+\infty)\rightarrow[0,+\infty)$  is called a {\itshape pre-weight function,} if it is continuous, non-decreasing, $\omega(0)=0$ and such that
	\begin{itemize}
		\item[$(*)$] $\log(t)=o(\omega(t))$, $t\rightarrow+\infty$,
		
		\item[$(*)$] $t\mapsto\varphi_{\omega}(t):=\omega(e^t)$ is convex.
	\end{itemize}
	Consequently, for each pre-weight function we have $\lim_{t\rightarrow+\infty}\omega(t)=+\infty$. $\omega$ is called a {\itshape weight function} if $\omega$ satisfies in addition
	\begin{itemize}
		\item[$(*)$] $\omega(2t)=O(\omega(t))$, $t\rightarrow+\infty$.
	\end{itemize}
	
	If $\omega(t)=0$ for all $t\in[0,1]$, then we call $\omega$ a {\itshape normalized (pre-)weight function}.

	Let $\sigma,\tau$ be pre-weight functions, we write $\sigma\hypertarget{ompreceq}{\preceq}\tau$ if $\tau(t)=O(\sigma(t))\;\text{as}\;t\rightarrow+\infty$. We call them equivalent, denoted by $\sigma\hypertarget{sim}{\sim}\tau$, if
	$\sigma\hyperlink{ompreceq}{\preceq}\tau$ and $\tau\hyperlink{ompreceq}{\preceq}\sigma$.
	
	\subsection{Associated weight function}
	Let $M\in\RR_{>0}^{\NN}$ (with $M_0=1$), then the {\itshape associated function} $\omega_M: \RR_{\ge 0}\rightarrow\RR\cup\{+\infty\}$ is defined by
	\begin{equation*}\label{assofunc}
		\omega_M(t):=\sup_{j\in\NN}\log\left(\frac{t^j}{M_j}\right)\;\;\;\text{for}\;t\in \RR_{>0},\hspace{30pt}\omega_M(0):=0.
	\end{equation*}
	For an abstract introduction of the associated function we refer to \cite[Chapitre I]{mandelbrojtbook}; see also \cite[Definition 3.1]{Komatsu73}.
	
	If $\liminf_{j\rightarrow+\infty}(M_j)^{1/j}>0$, then $\omega_M(t)=0$ for sufficiently small $t$, since $\log\left(\frac{t^j}{M_j}\right)<0\Leftrightarrow t<(M_j)^{1/j}$ holds for all $j\in\NN_{>0}$. (In particular, if $M_j\ge 1$ for all $j\in\NN$, then $\omega_M$ is vanishing on $[0,1]$.) Moreover, under this assumption $t\mapsto\omega_M(t)$ is a continuous nondecreasing function, which is convex in the variable $\log(t)$ and tends faster to infinity than any $\log(t^j)$, $j\ge 1$, as $t\rightarrow+\infty$. $\lim_{j\rightarrow+\infty}(M_j)^{1/j}=+\infty$ implies that $\omega_M(t)<+\infty$ for each finite $t$ which shall be considered as a basic assumption for defining $\omega_M$.\vspace{6pt}
	
	Summarizing, if $M\in\hyperlink{LCset}{\mathcal{LC}}$, then $\omega_M$ is a normalized pre-weight function (e.g. see \cite[Lemma 3.1]{sectorialextensions1}), however in general $\omega_M(2t)=O(\omega_M(t))$ is not clear; see the recent characterization \cite[Thm. 3.1]{subaddlike}.
	
	Finally, if $M\in\hyperlink{LCset}{\mathcal{LC}}$ (or even if $M$ is log-convex with $M_0=1$ and $\lim_{j\rightarrow+\infty}(M_j)^{1/j}=+\infty$), then by \cite[Chapitre I, 1.4, 1.8]{mandelbrojtbook} and also \cite[Prop. 3.2]{Komatsu73} we get
	\begin{equation}\label{Prop32Komatsu}
		M_j=\sup_{t\ge 0}\frac{t^j}{\exp(\omega_{M}(t))},\;\;\;j\in\NN.
	\end{equation}
	
	\subsection{Weight matrices}\label{weightmatrixsection}
	For the following definitions and conditions see \cite[Sect. 4]{compositionpaper} and \cite[Sect. 7]{dissertation}.
	
	Let $\mathcal{I}=\RR_{>0}$ denote the index set (equipped with the natural order). A {\itshape weight matrix} $\mathcal{M}$ associated with $\mathcal{I}$ is a (one parameter) family of weight sequences $\mathcal{M}:=\{M^{(\alpha)}\in\RR_{>0}^{\NN}: \alpha\in\mathcal{I}\}$, such that
	$$\forall\;\alpha\le\beta:\;\;\; M^{(\alpha)}\le M^{(\beta)}.$$
	We call a weight matrix $\mathcal{M}$ {\itshape standard log-convex,} denoted by \hypertarget{Msc}{$(\mathcal{M}_{\on{sc}})$}, if
	$$\forall\;\alpha\in\mathcal{I}:\;M^{(\alpha)}\in\hyperlink{LCset}{\mathcal{LC}}.$$
	Moreover, we put $m^{(\alpha)}_j:=\frac{M^{(\alpha)}_j}{j!}$ for $j\in\NN$, and $\mu^{(\alpha)}_j:=\frac{M^{(\alpha)}_j}{M^{(\alpha)}_{j-1}}$ for $j\in\NN_{>0}$, $\mu^{(\alpha)}_0:=1$.
	
	A matrix is called {\itshape constant} if $M^{(\alpha)}\hyperlink{approx}{\approx}M^{(\beta)}$ for all $\alpha,\beta\in\mathcal{I}$.\vspace{6pt}
	
	Let $\mathcal{M}=\{M^{(\alpha)}: \alpha\in\mathcal{I}\}$ and $\mathcal{N}=\{N^{(\alpha)}: \alpha\in\mathcal{I}\}$ be given. We write $\mathcal{M}\hypertarget{Mroumpreceq}{\{\preceq\}}\mathcal{N}$ if
	$$\forall\;\alpha\in\mathcal{I}\;\exists\;\beta\in\mathcal{I}:\;\;\;M^{(\alpha)}\hyperlink{preceq}{\preceq}N^{(\beta)},$$
	and call $\mathcal{M}$ and $\mathcal{N}$
	to be $R$-equivalent, or $\mathcal{M} \{ \approx \} \mathcal{N}$ for short, if
	$\mathcal{M}\hyperlink{Mroumpreceq}{\{\preceq\}}\mathcal{N}$ and $\mathcal{N}\hyperlink{Mroumpreceq}{\{\preceq\}}\mathcal{M}$. Analogously, we write $\mathcal{M}\hypertarget{Mbeurpreceq}{(\preceq)}\mathcal{N}$ if
	$$\forall\;\alpha\in\mathcal{I}\;\exists\;\beta\in\mathcal{I}:\;\;\;M^{(\beta)}\hyperlink{preceq}{\preceq}N^{(\alpha)},$$
	and call $\mathcal{M}$ and $\mathcal{N}$ to be $B$-equivalent, or $\mathcal{M} ( \approx ) \mathcal{N}$ for short, if
	$\mathcal{M}\hyperlink{Mbeurpreceq}{(\preceq)}\mathcal{N}$ and $\mathcal{N}\hyperlink{Mbeurpreceq}{(\preceq)}\mathcal{M}$.
	
	If $\mathcal{M}$ and $\mathcal{N}$ are both $R$- and $B$-equivalent, then we say for simplicity that they are equivalent.\vspace{6pt}
	
	We recall several growth and regularity assumptions on a given weight matrix:
	\begin{align*}
		&\hypertarget{holom}{(\mathcal{M}_{\mathcal{H}})} & & \forall\;\alpha\in\mathcal{I}:\;\;\;\liminf_{j\rightarrow\infty}(m^{(\alpha)}_j)^{1/j}>0,\\
		&\hypertarget{R-Comega}{(\mathcal{M}_{\{\text{C}^{\omega}\}})} &&\exists\;\alpha\in\mathcal{I}:\;\;\;\liminf_{j\rightarrow\infty}(m^{(\alpha)}_j)^{1/j}>0,\\
		&\hypertarget{B-Comega}{(\mathcal{M}_{(\text{C}^{\omega})})}& & \forall\;\alpha\in\mathcal{I}:\;\;\;\lim_{j\rightarrow\infty}(m^{(\alpha)}_j)^{1/j}=+\infty,\\
		&\hypertarget{R-rai}{(\mathcal{M}_{\{\text{rai}\}})}&& \forall\;\alpha\in\mathcal{I}\;\exists\;C>0\;\exists\;\beta\in\mathcal{I}\;\forall\;1\le j\le k:\;\;\;(m^{(\alpha)}_j)^{1/j}\le C(m^{(\beta)}_k)^{1/k},\\
		&\hypertarget{B-rai}{(\mathcal{M}_{(\text{rai})})} && \forall\;\alpha\in\mathcal{I}\;\exists\;C>0\;\exists\;\beta\in\mathcal{I}\;\forall\;1\le j\le k:\;\;\;(m^{(\beta)}_j)^{1/j}\le C(m^{(\alpha)}_k)^{1/k},\\
		&\hypertarget{R-mg}{(\mathcal{M}_{\{\text{mg}\}})} && \forall\;\alpha\in\mathcal{I}\;\exists\;C>0\;\exists\;\beta\in\mathcal{I}\;\forall\;j,k\in\NN:\;\;\;M^{(\alpha)}_{j+k}\le C^{j+k+1} M^{(\beta)}_j M^{(\beta)}_k,\\
		&\hypertarget{B-mg}{(\mathcal{M}_{(\text{mg})})}&& \forall\;\alpha\in\mathcal{I}\;\exists\;C>0\;\exists\;\beta\in\mathcal{I}\;\forall\;j,k\in\NN:\;\;\;M^{(\beta)}_{j+k}\le C^{j+k+1} M^{(\alpha)}_j M^{(\alpha)}_k,\\
		&\hypertarget{R-dc}{(\mathcal{M}_{\{\text{dc}\}})} && \forall\;\alpha\in\mathcal{I}\;\exists\;C>0\;\exists\;\beta\in\mathcal{I}\;\forall\;j\in\NN:\;\;\;M^{(\alpha)}_{j+1}\le C^{j+1} M^{(\beta)}_j,\\
		&\hypertarget{B-dc}{(\mathcal{M}_{(\text{dc})})}&& \forall\;\alpha\in\mathcal{I}\;\exists\;C>0\;\exists\;\beta\in\mathcal{I}\;\forall\;j\in\NN:\;\;\;M^{(\beta)}_{j+1}\le C^{j+1} M^{(\alpha)}_j,\\
		&\hypertarget{R-BR}{(\mathcal{M}_{\{\text{BR}\}})} & & \forall\;\alpha\in\mathcal{I}\;\exists\;\beta\in\mathcal{I}:\;\;\; M^{(\alpha)} \hyperlink{triangle}{\vartriangleleft}  M^{(\beta)} ,\\
		&\hypertarget{B-BR}{(\mathcal{M}_{(\text{BR})})} && \forall\;\alpha\in\mathcal{I}\;\exists\;\beta\in\mathcal{I}:\;\;\;  M^{(\beta)} \hyperlink{triangle}{\vartriangleleft}  M^{(\alpha)} ,\\
		&\hypertarget{R-FdB}{(\mathcal{M}_{\{\text{FdB}\}})} & & \forall\;\alpha\in\mathcal{I}\;\exists\;\beta\in\mathcal{I}:\;\;\; (m^{(\alpha)})^{\circ}\hyperlink{mpreceq}{\preceq} m^{(\beta)},\\
		&\hypertarget{B-FdB}{(\mathcal{M}_{(\text{FdB})})} && \forall\;\alpha\in\mathcal{I}\;\exists\;\beta\in\mathcal{I}:\;\;\; (m^{(\beta)})^{\circ}\hyperlink{mpreceq}{\preceq} m^{(\alpha)},
	\end{align*}
	with $(m^{(\alpha)})^{\circ}:=((m^{(\alpha)}_j)^{\circ})_j$ being the sequence defined by
	$$(m^{(\alpha)}_k)^{\circ}:=\max\left\{m^{(\alpha)}_{\ell}\cdot m^{(\alpha)}_{j_1}\cdots m^{(\alpha)}_{j_{\ell}}: j_i\in\NN_{>0}, \sum_{i=1}^{\ell}j_i=k\right\},\;\;\;\;\;(m^{(\alpha)}_0)^{\circ}:=1.$$

	$R$-equivalence between matrices preserves all Roumieu-type conditions listed above and $B$-equivalence all Beurling-type conditions.
	
	Finally, let us recall
	\begin{align*}
		&\hypertarget{R-L}{(\mathcal{M}_{\{\text{L}\}})} && \forall\;C>0\;\forall\;\alpha\in\mathcal{I}\;\exists\;D>0\;\exists\;\beta\in\mathcal{I}\;\forall\;j\in\NN: C^j M^{(\alpha)}_j\le D M^{(\beta)}_j,\\
		&\hypertarget{B-L}{(\mathcal{M}_{(\text{L})})}&& \forall\;C>0\;\forall\;\alpha\in\mathcal{I}\;\exists\;D>0\;\exists\;\beta\in\mathcal{I}\;\forall\;j\in\NN: C^j M^{(\beta)}_j\le D M^{(\alpha)}_j.
	\end{align*}
	
	A matrix is called {\itshape non-quasianalytic} if any sequence $M^{(\alpha)}$ is non-quasianalytic; see \cite[Sect. 4]{testfunctioncharacterization}. When dealing with Roumieu type classes then it suffices to assume that there exists $\alpha_0\in\mathcal{I}$ such that $M^{(\alpha_0)}$ is non-quasianalytic since smaller indices can be skipped; see also the discussion in \cite[Sect. 5.1]{borelmappingquasianalytic}.
	
	\subsection{Ultradifferentiable classes}\label{ultradiffclassessection}
	Let $U\subseteq\RR^d$ be non-empty open and for $K\subseteq\RR^d$ compact we write $K\subset\subset U$ if $\overline{K}\subseteq U$, i.e., $K$ is in $U$ relatively compact. We introduce now the following spaces of ultradifferentiable function classes.
	First, for weight sequences we define the (local) classes of {\itshape Roumieu-type} by
	$$\mathcal{E}_{\{M\}}(U):=\{f\in\mathcal{E}(U):\;\forall\;K\subset\subset U\;\exists\;h>0:\;\|f\|_{M,K,h}<+\infty\},$$
	and the classes of {\itshape Beurling-type} by
	$$\mathcal{E}_{(M)}(U):=\{f\in\mathcal{E}(U):\;\forall\;K\subset\subset U\;\forall\;h>0:\;\|f\|_{M,K,h}<+\infty\},$$
	where we denote
	\begin{equation*}\label{semi-norm-2}
		\|f\|_{M,K,h}:=\sup_{\alpha\in\NN^d,x\in K}\frac{|f^{(\alpha)}(x)|}{h^{|\alpha|} M_{|\alpha|}}.
	\end{equation*}
	
	For a sufficiently regular compact set $K$ (e.g. with smooth boundary and such that $\overline{K^\circ}=K$)
	$$\mathcal{E}_{M,h}(K):=\{f\in\mathcal{E}(K): \|f\|_{M,K,h}<+\infty\}$$
	is a Banach space and so we have the following topological vector spaces
	$$\mathcal{E}_{\{M\}}(K):=\underset{h>0}{\varinjlim}\;\mathcal{E}_{M,h}(K),$$
	and
	\begin{equation*}
		\mathcal{E}_{\{M\}}(U)=\underset{K\subset\subset U}{\varprojlim}\;\underset{h>0}{\varinjlim}\;\mathcal{E}_{M,h}(K)=\underset{K\subset\subset U}{\varprojlim}\;\mathcal{E}_{\{M\}}(K).
	\end{equation*}
	Similarly, we get
	$$\mathcal{E}_{(M)}(K):=\underset{h>0}{\varprojlim}\;\mathcal{E}_{M,h}(K),$$
	and
	\begin{equation*}
		\mathcal{E}_{(M)}(U)=\underset{K\subset\subset U}{\varprojlim}\;\underset{h>0}{\varprojlim}\;\mathcal{E}_{M,h}(K)=\underset{K\subset\subset U}{\varprojlim}\;\mathcal{E}_{(M)}(K).
	\end{equation*}
	
	For a weight function, we define the corresponding Roumieu-type classes
	$$\mathcal{E}_{\{\omega\}}(U):=\{f\in\mathcal{E}(U):\;\forall\;K\subset\subset U\;\exists\;h>0:\;\|f\|_{\omega,K,h}<+\infty\},$$
	and the classes of Beurling-type by
	$$\mathcal{E}_{(\omega)}(U):=\{f\in\mathcal{E}(U):\;\forall\;K\subset\subset U\;\forall\;h>0:\;\|f\|_{\omega,K,h}<+\infty\},$$
	where we denote
	\begin{equation*}\label{semi-norm-2}
		\|f\|_{\omega,K,h}:=\sup_{\alpha\in\NN^d,x\in K}\frac{|f^{(\alpha)}(x)|}{\exp(\frac{1}{h} \varphi_\omega^*(h|\alpha|))}.
	\end{equation*}
	The spaces are topologized in complete analogy to the weight sequence case. First we define for a sufficiently regular compact set $K$ the Banach space
	$$\mathcal{E}_{\omega,h}(K):=\{f\in\mathcal{E}(K): \|f\|_{\omega,K,h}<+\infty\},$$
	and set
	\[
	\mathcal{E}_{\{\omega\}}(K):=\underset{h>0}{\varinjlim}\;\mathcal{E}_{\omega,h}(K), \quad \mathcal{E}_{(\omega)}(K):=\underset{h>0}{\varprojlim}\;\mathcal{E}_{\omega,h}(K),
	\]
	finally we endow $\mathcal{E}_{[\omega]}(U)$ with the following locally convex topologies
	\[
	\mathcal{E}_{\{\omega\}}(U)=\underset{K\subset\subset U}{\varprojlim}\;\mathcal{E}_{\{\omega\}}(K), \quad \mathcal{E}_{(\omega)}(U)=\underset{K\subset\subset U}{\varprojlim}\;\mathcal{E}_{(\omega)}(K).
	\]

	Next, we consider classes defined by weight matrices of Roumieu-type $\mathcal{E}_{\{\mathcal{M}\}}$ and of Beurling-type $\mathcal{E}_{(\mathcal{M})}$ as follows; see also \cite[4.2]{compositionpaper}. For a weight matrix $\mathcal{M}=\{M^{(x)}:x\in \mathcal{I} \}$ and a sufficiently regular $K\subset\subset U$ we put
	\begin{equation}\label{generalroumieu}
		\mathcal{E}_{\{\mathcal{M}\}}(K):=\bigcup_{x\in\mathcal{I}}\mathcal{E}_{\{M^{(x)}\}}(K),\hspace{20pt}\mathcal{E}_{\{\mathcal{M}\}}(U):=\bigcap_{K\subset\subset U}\bigcup_{x\in\mathcal{I}}\mathcal{E}_{\{M^{(x)}\}}(K),
	\end{equation}
	and
	\begin{equation*}\label{generalbeurling}
		\mathcal{E}_{(\mathcal{M})}(K):=\bigcap_{x\in\mathcal{I}}\mathcal{E}_{(M^{(x)})}(K),\hspace{20pt}\mathcal{E}_{(\mathcal{M})}(U):=\bigcap_{x\in\mathcal{I}}\mathcal{E}_{(M^{(x)})}(U).
	\end{equation*}
	For such $K$ one has the representation
	$$\mathcal{E}_{\{\mathcal{M}\}}(K)=\underset{x\in\mathcal{I}}{\varinjlim}\;\underset{h>0}{\varinjlim}\;\mathcal{E}_{M^{(x)},h}(K)$$
	and so for $U\subseteq\RR^d$ non-empty open
	\begin{equation*}\label{generalroumieu1}
		\mathcal{E}_{\{\mathcal{M}\}}(U)=\underset{K\subset\subset U}{\varprojlim}\;\underset{x\in\mathcal{I}}{\varinjlim}\;\underset{h>0}{\varinjlim}\;\mathcal{E}_{M^{(x)},h}(K).
	\end{equation*}
	Similarly, we get for the Beurling case
	\begin{equation*}
		\mathcal{E}_{(\mathcal{M})}(U)=\underset{K\subset\subset U}{\varprojlim}\;\underset{x\in\mathcal{I}}{\varprojlim}\;\underset{h>0}{\varprojlim}\;\mathcal{E}_{M^{(x)},h}(K).
	\end{equation*}
	
	\subsection{Pilipovi\'{c}-Teofanov-Tomi\'{c} classes}
	\label{sec:ptt1}
	Let $\tau,h>0$, $\sigma\ge 1$, then one considers the weight sequence
	\[
	M^{\tau,\sigma}_j:= j^{\tau j^\sigma}.
	\]
	Let $K\subset \subset \RR^d$ be a sufficiently regular compact set. By $\mathcal{E}_{\tau,\sigma,h}(K)$ we shall denote the Banach space of functions $\phi\in\mathcal{E}(K)$ such that
	$$||\phi||_{\mathcal{E}_{\tau,\sigma,h}(K)} =\sup_{\alpha\in\NN^d}\sup_{x\in K} \frac{|\phi^{(\alpha)}(x)|}{h^{|\alpha|^\sigma} M_{|\alpha|}^{\tau, \sigma} }<+\infty.$$
	Note that the case $\sigma=1$ gives, by Stirling's formula, that $M^{\tau,1}\hyperlink{approx}{\approx}G^{\tau}$, i.e., $M^{\tau,1}$ is equivalent to the classical Gevrey sequence with parameter/index $\tau$.
	
	If $0<h_1\leq h_2$, $0 < \tau_1 \leq \tau_2$, $1\leq\sigma_1\leq \sigma_2$, then
	$\mathcal{E}_{\tau_1,\sigma_1,h_1}(K) \hookrightarrow \mathcal{E}_{\tau_2,\sigma_2,h_2}(K) $.
	
	Let $U$ be an open set of $\RR^d$. We define the spaces:
	$$ \mathcal{E}_{\{\tau,\sigma\}} (U) = \varprojlim_{K\subset \subset U } \varinjlim_{h\to\infty} \mathcal{E}_{\tau,\sigma,h}(K)  \qquad \qquad \mathcal{E}_{ (\tau,\sigma )} (U) = \varprojlim_{K\subset \subset U } \varprojlim_{h\to 0} \mathcal{E}_{\tau,\sigma,h}(K) $$
	$$ \mathcal{E}_{\{\infty,\sigma\}} (U) =  \varinjlim_{\tau\to\infty} \mathcal{E}_{\{\tau,\sigma\}}(U)  \qquad \qquad   \mathcal{E}_{ (\infty,\sigma )} (U) = \varinjlim_{\tau\to\infty} \mathcal{E}_{(\tau,\sigma)}(U) $$
	$$ \mathcal{E}_{\{0,\sigma\}} (U) =  \varprojlim_{\tau\to 0} \mathcal{E}_{\{\tau,\sigma\}}(U)  \qquad \qquad   \mathcal{E}_{ (0,\sigma )} (U) = \varprojlim_{\tau\to 0} \mathcal{E}_{(\tau,\sigma)}(U) $$
	We use the abbreviated notation $[\tau,\sigma]$ for $\{\tau,\sigma\}$  or $(\tau,\sigma)$.
	
	\begin{proposition} \cite[Prop. 2.1]{PTT2015} \cite[Prop. 2.1]{PTT2016}  Let $\sigma\geq 1$ and $\tau>0$. Then for every $\sigma_2>\sigma_1\geq 1$, we have that
		$$ \mathcal{E}_{ [\infty,\sigma_1]} (U)  \hookrightarrow \mathcal{E}_{ [0,\sigma_2 ]} (U). $$
		Moreover, if $0<\tau_1<\tau_2$ then
		$$ \mathcal{E}_{\{\tau_1,\sigma\}} (U)   \hookrightarrow  \mathcal{E}_{ (\tau_2,\sigma )} (U)  \hookrightarrow  \mathcal{E}_{\{\tau_2,\sigma\}} (U), $$
		and
		$$\mathcal{E}_{\{\infty,\sigma\}} (U) = \mathcal{E}_{ (\infty,\sigma )} (U) \qquad \qquad  \mathcal{E}_{\{0,\sigma\}} (U) = \mathcal{E}_{ (0,\sigma )} (U). $$
		Consequently, given $\tau_0>0$ we see that
		$$\mathcal{E}_{[\tau_0,\sigma_1]} (U)  \hookrightarrow \bigcap_{\tau>\tau_0} \mathcal{E}_{[\tau,\sigma_1]} (U)  \hookrightarrow  \mathcal{E}_{[\tau_0,\sigma_2]} (U). $$
		In particular, if $\sigma > 1$ then
		$$\mathcal{E}_{[\infty,1]} (U)  \hookrightarrow \mathcal{E}_{[\tau,\sigma]} (U).$$
	\end{proposition}
	
	\section{Ultradifferentiable classes beyond geometric growth factors}\label{ultrabeyondsection}
	
	The main objective is to prove that the PTT-classes can be represented as classes defined by (suitable) weight matrices. Indeed, we can obtain a more general result by letting the exponents of the defining estimates be $\Phi_j$ instead of  ${j}$ or $j^\sigma$ where $\Phi = (\Phi_j)_j$ is arbitrary and only satisfying some mild regularity property.
	
	\begin{definition}
		A sequence $\Phi \in \RR_{\ge 0}^{\NN}$ is called \emph{exponent sequence} if it satisfies
		\begin{equation}\label{Philiminf}
			\liminf_{j\rightarrow\infty}\frac{\Phi_j}{j}>0.
		\end{equation}
	\end{definition}
	
	In particular every exponent sequence tends to infinity.

	Let $U\subseteq\RR^d$ be non-empty open, $M\in\RR_{>0}^{\NN}$ and $\Phi \in \RR_{\ge 0}^{\NN}$. We introduce now $\Phi$-ultradifferentiable function classes (defined in terms of a single weight sequence $M$): The (local) class of Roumieu-type is given by
	$$\mathcal{E}_{\{M,\Phi\}}(U):=\{f\in\mathcal{E}(U):\;\forall\;K\subset\subset U\;\exists\;h>0:\;\|f\|_{M,\Phi,K,h}<+\infty\},$$
	and the class of Beurling-type by
	$$\mathcal{E}_{(M,\Phi)}(U):=\{f\in\mathcal{E}(U):\;\forall\;K\subset\subset U\;\forall\;h>0:\;\|f\|_{M,\Phi,K,h}<+\infty\},$$
	where we denote
	\begin{equation}\label{Phiseminorm}
		\|f\|_{M,\Phi,K,h}:=\sup_{\alpha\in\NN^d,x\in K}\frac{|f^{(\alpha)}(x)|}{h^{\Phi_{|\alpha|}} M_{|\alpha|}}.
	\end{equation}
	
	For a sufficiently regular compact set $K$
	$$\mathcal{E}_{M,\Phi,h}(K):=\{f\in\mathcal{E}(K): \|f\|_{M,\Phi,K,h}<+\infty\}$$
	is a Banach space and so we have the following topological vector space representations
	$$\mathcal{E}_{\{M,\Phi\}}(K):=\underset{h>0}{\varinjlim}\;\mathcal{E}_{M,\Phi,h}(K), \qquad \text{and} \qquad
	\mathcal{E}_{\{M,\Phi\}}(U)=\underset{K\subset\subset U}{\varprojlim}\;\underset{h>0}{\varinjlim}\;\mathcal{E}_{M,\Phi,h}(K)=\underset{K\subset\subset U}{\varprojlim}\;\mathcal{E}_{\{M,\Phi\}}(K).$$
	Similarly, we get
	$$\mathcal{E}_{(M,\Phi)}(K):=\underset{h>0}{\varprojlim}\;\mathcal{E}_{M,\Phi,h}(K),
	\qquad \text{and} \qquad
	\mathcal{E}_{(M,\Phi)}(U)=\underset{K\subset\subset U}{\varprojlim}\;\underset{h>0}{\varprojlim}\;\mathcal{E}_{M,\Phi,h}(K)=\underset{K\subset\subset U}{\varprojlim}\;\mathcal{E}_{(M,\Phi)}(K).$$
	
		Note that condition \eqref{Philiminf} means that the factor $h^{\Phi_{|\alpha|}}$ in the seminorms is at least geometric.

	\begin{remark}\label{otherclassesremark} In the literature, global ultradifferentiable classes, test function spaces, ultraholomorphic classes, spaces of weighted sequences of complex numbers and,  PTT-test function spaces (see \cite{PTT2015})  have been  studied.
	In a completely analogous way, global $\Phi$-ultradifferentiable classes, $\Phi$-test function spaces, $\Phi$-ultraholomorphic classes and $\Phi$-spaces of weighted sequences of complex numbers (weighted with seminorms of the type \eqref{Phiseminorm}) can be defined ; i.e., when the symbol/functor $\mathcal{E}$  is replaced by $\mathcal{B}$, $\mathcal{D}$, $\mathcal{A}$ or $\Lambda$.
	\end{remark}
	
	\textbf{Notation:} When the open set $U$ is not relevant in certain statement, we will simply write $\mathcal{E}_{\{\mathcal{M}\}}$, $\mathcal{E}_{(\mathcal{M})}$,  $\mathcal{E}_{\{\tau,\sigma\}}$, $\mathcal{E}_{ (\tau,\sigma )}$ or  $\mathcal{E}_{\{M,\Phi\}}$, $\mathcal{E}_{(M,\Phi)}$.

	\begin{example}\label{PPTexamples}
		We have two important examples in mind:
		\begin{itemize}
			\item[$(a)$] If $\Phi_j=j$ for all $j\in \NN$, then we recover the classes from Section \ref{ultradiffclassessection}; i.e., the usual definition of ultradifferentiable classes defined in terms of weight sequences; e.g. see \cite{Komatsu73}. For reasons of simplicity we will skip the letter $\Phi$ in the definition.
			
			\item[$(b)$] Let parameters $\sigma>1$ and $\tau>0$ be given. Then the choices
			\begin{equation}\label{PPTsequence}
				M^{\tau,\sigma}_j:=j^{\tau j^{\sigma}},\hspace{15pt}\Phi_j:=j^{\sigma},\;\;\;j\in\NN,
			\end{equation}
			yield the classes introduced in \cite{PTT2015}. We use the convention $0^0:=1$ and so $M^{\tau,\sigma}_0=1$.
		\end{itemize}
	\end{example}
	
	
	Given an exponent sequence $\Phi\in \RR_{\ge0}^{\NN}$ and a pair of sequences  $M,N\in\RR_{>0}^{\NN}$,  we want to compare the classes $\mathcal{E}_{\{M,\Phi\}}$ and $\mathcal{E}_{\{N,\Phi\}}$, or resp.  $\mathcal{E}_{(M,\Phi)}$ and $\mathcal{E}_{(N,\Phi)}$. We write $M\hypertarget{preceq}{\preceq_{\Phi}}N$ if $$\sup_{j\in\NN_{>0}}\left(\frac{M_j}{N_j}\right)^{1/\Phi_j}<+\infty,$$
	and we call $M$ and $N$ to be $\Phi$-{\itshape equivalent}, written as $M\hypertarget{approx}{\approx_{\Phi}}N$, if $M\hyperlink{preceq}{\preceq_{\Phi}}N$ and $N\hyperlink{preceq}{\preceq_{\Phi}}M$.
	
	By definition obviously $M\hyperlink{preceq}{\preceq_{\Phi}}N$ implies both $\mathcal{E}_{\{M,\Phi\}}\subseteq\mathcal{E}_{\{N,\Phi\}}$ and $\mathcal{E}_{(M,\Phi)}\subseteq\mathcal{E}_{(N,\Phi)}$ with continuous inclusion. Thus $\Phi$-equivalent sequences define the same associated function classes; i.e., if $M\hyperlink{approx}{\approx_{\Phi}}N$ then $\mathcal{E}_{\{M,\Phi\}}=\mathcal{E}_{\{N,\Phi\}}$ and $\mathcal{E}_{(M,\Phi)}=\mathcal{E}_{(N,\Phi)}$ (as topological vector spaces) and similarly for the other classes mentioned in Remark \ref{otherclassesremark}. Note that for $\Phi_j=j$ relation \hyperlink{preceq}{$\preceq_{\Phi}$} is precisely \hyperlink{mpreceq}{$\preceq$}; see \cite[p. 101, Prop. 2.12 $(1)$]{compositionpaper}.
	
	Analogously, we write $M\hypertarget{triangle}{\vartriangleleft_{\Phi}}N$ if $$\lim_{j\rightarrow+\infty}\left(\frac{M_j}{N_j}\right)^{1/\Phi_j}=0,$$
	which is obviously never reflexive and symmetric and stronger than \hyperlink{preceq}{$\preceq_{\Phi}$}. If $M\hyperlink{triangle}{\vartriangleleft_{\Phi}}N$, then $\mathcal{E}_{\{M,\Phi\}}\subseteq\mathcal{E}_{(N,\Phi)}$ with continuous inclusion (and similarly for the other classes mentioned in Remark \ref{otherclassesremark}).\vspace{6pt}
	
	
	\section{Classes beyond geometric factors versus weight matrices}\label{ultrabeyondvsweightmatrixsection}
	The aim of this section is to verify that, under mild growth and regularity assumptions on $\Phi$, the  classes $\mathcal{E}_{\{M,\Phi\}}$ resp. $\mathcal{E}_{(M,\Phi)}$ can be represented (as locally convex vector spaces) by the  matrix classes $\mathcal{E}_{\{\mathcal{M}\}}$ resp. $\mathcal{E}_{(\mathcal{M})}$ for a suitable choice of the matrix $\mathcal{M}$.
	
	\subsection{Preparatory results}
	We introduce an appropriate matrix of sequences. Let $M\in\RR_{>0}^{\NN}$ and $\Phi\in \RR_{\ge0}^{\NN}$, then consider the set
	\begin{equation}\label{convenientmatrix}
		\mathcal{M}_{M,\Phi}:=\{M^{(c,\Phi)}: c>0\},\hspace{15pt}M^{(c,\Phi)}_j:=c^{\Phi_j}M_j,\;\;\;j\in\NN.
	\end{equation}
	For any sequence $\Phi\in\RR_{\ge0}^{\NN}$, we have that $\mathcal{M}_{M,\Phi}$ is a weight matrix in the notion of Section \ref{weightmatrixsection}. Next we show that the mild growth restriction  \eqref{Philiminf} on $\Phi$ is equivalent to the fact that the matrix  $\mathcal{M}_{M,\Phi}$ allows to absorb exponential growth.
	
	\begin{lemma}\label{lemma.MLCondition} Let $M\in\RR_{>0}^{\NN}$ and $\Phi\in\RR_{\ge0}^{\NN}$ be given. Then the following are equivalent:
		\begin{enumerate}
			\item[$(i)$] $\Phi$ satisfies \eqref{Philiminf}, i.e., $\Phi$ is an exponent sequence.
			\item[$(ii)$] $\mathcal{M}_{M,\Phi}$ satisfies \hyperlink{R-L}{$(\mathcal{M}_{\{\on{L}\}})$}.
			\item[$(iii)$] $\mathcal{M}_{M,\Phi}$ satisfies \hyperlink{B-L}{$(\mathcal{M}_{(\on{L})})$}.
		\end{enumerate}
	\end{lemma}
	\begin{proof} $(i)\Leftrightarrow (ii)$ First we observe that $\mathcal{M}_{M,\Phi}$ satisfies \hyperlink{R-L}{$(\mathcal{M}_{\{\text{L}\}})$} if and only if
		$$\forall\;h>0\;\forall\;c>0\;\exists\;D>0\;\exists\; c_1>0\;\forall\;j\in\NN: h^j c^{\Phi_j} \le D c_1^{\Phi_j}.$$
		Consequently, if (ii) holds,  then for $h=2$, $c=1$ there exists
		$c_1 > 1 $  and $D>0$ such that
		$$\frac{\Phi_j}{j} \geq \frac{\log(2)-(1/j)\log(D)}{\log(c_1)},$$
		for all $j\in\NN_{>0}$, so $\Phi$ satisfies \eqref{Philiminf}.
		
		Conversely, let $h,c\ge 1$ be given. Then \eqref{Philiminf} yields the existence of some $\epsilon>0$ and $j_{\epsilon}\in\NN_{>0}$ such that for all $j\ge j_{\epsilon}$ we get $\frac{\Phi_j}{j}\ge\epsilon$. So there exists some $c_1>c$ such that $\left(\frac{c_1}{c}\right)^{\Phi_j/j}\ge\left(\frac{c_1}{c}\right)^{\epsilon}\ge h$ for all $j\ge j_{\epsilon}$. Thus, when choosing $A\ge 1$ large enough, then we have that for all $j\in\NN$
		$$h^jM^{(c,\Phi)}_j=h^jc^{\Phi_j}M_j\le Ac_1^{\Phi_j}M_j=AM^{(c_1,\Phi)}_j,$$
		which shows that  $\mathcal{M}_{M,\Phi}$ satisfies \hyperlink{R-L}{$(\mathcal{M}_{\{\text{L}\}})$}.

		$(i) \Leftrightarrow (iii)$ Follows similarly.
	\end{proof}
	
			
			
		

	If we assume more growth requirements on the sequence $\Phi$, then we can deduce further regularity conditions for $\mathcal{M}_{M,\Phi}$.
	
	\begin{proposition}\label{Phigrowthprop}
		Let $M\in\RR_{>0}^{\NN}$ and $\Phi\in\RR_{\ge0}^{\NN}$ be given.
		\begin{itemize}
			\item[$(i)$] If $M$ is normalized and if $\Phi_0=\Phi_1=0$, then each $M^{(c,\Phi)}$ is normalized, too.

			\item[$(ii)$] We observe that $M^{(c,\Phi)}$ is log-convex if and only if
			\begin{equation}\label{McPhilogconv}
				j\mapsto\mu^{(c,\Phi)}_j=c^{\Phi_j-\Phi_{j-1}}\mu_j \qquad \text{is non-decreasing.}
			\end{equation}
			In particular, if $M$ is log-convex, i.e., $j\mapsto\mu_j$ is non-decreasing, and if $\Phi$ is convex,  i.e.,
			$$\forall\;j\in\NN_{>0}:\;\;\;2\Phi_j\leq \Phi_{j-1} + \Phi_{j+1},$$
			then $M^{(c,\Phi)}$ is log-convex for each $c\ge 1$.
			
				Moreover, if $\Phi$ is  convex and \eqref{McPhilogconv} holds for some $c>0$, then also for all $d> c$.

			\item[$(iii)$] If $\Phi$ is increasing, convex, and $\Phi_0=0$, then it is an exponent sequence. Moreover, since $\Phi_{j}-\Phi_{j-1}\ge 0$, we get
			\begin{equation*}\label{mumatrixphi}
				\forall\;0<c_1\le c_2\;\forall\;j\ge 1:\;\;\;\mu^{(c_1,\Phi)}_j=\frac{M^{(c_1,\Phi)}_j}{M^{(c_1,\Phi)}_{j-1}}
				=c_1^{\Phi_j-\Phi_{j-1}}\mu_j\le c_2^{\Phi_{j}-\Phi_{j-1}}\mu_j=\mu^{(c_2,\Phi)}_j,
			\end{equation*}
			i.e., the sequences are even ordered w.r.t. their corresponding quotient sequences.

			\item[$(iv)$]  If $(M_j)^{1/j}\rightarrow\infty$ as $j\rightarrow\infty$, and $\Phi$ is an exponent sequence, then
			$$\forall\;c\ge 1:\;\;(M^{(c,\Phi)}_j)^{1/j}= c^{\Phi_j/j}(M_j)^{1/j}\rightarrow\infty,\;j\rightarrow\infty.$$
			

		\end{itemize}
	\end{proposition}

	\begin{proof}
		$(i)$ and $(iv)$ are direct consequences. For the last statement in $(ii)$ let $d>c$ and write
$$d^{\Phi_j-\Phi_{j-1}}\mu_j=c^{\Phi_j-\Phi_{j-1}}\mu_j\left(\frac{d}{c}\right)^{\Phi_j-\Phi_{j-1}}.$$
Since $\frac{d}{c}\ge 1$ and the convexity of $\Phi$ precisely means that $j\mapsto\Phi_j-\Phi_{j-1}$ is non-decreasing we get that $j\mapsto d^{\Phi_j-\Phi_{j-1}}\mu_j$ is non-decreasing as well.\vspace{6pt}

Let us give an argument for $(iii)$:    From convexity since $\Phi_0=0$, one can deduce that $j \mapsto \frac{\Phi_j}{j}$ is non-decreasing:\par
		Observe that
		\[
		\Phi_{j}-\Phi_{j-1} \ge \frac{\Phi_{j-1}}{j-1},
		\]
		and therefore
		\[
		\Phi_j=(j-1)\frac{\Phi_{j-1}}{j-1} + \Phi_j-\Phi_{j-1}\ge j \frac{\Phi_{j-1}}{j-1}.
		\]
		Since $\Phi$ is increasing we get $\Phi_j>0$ for all large $j$ and hence $\frac{\Phi_j}{j}$ is bounded away from $0$.
	\end{proof}


	\begin{remark}\label{log-remark}
		Since log-convexity for $M^{(c,\Phi)}$ and $(M^{(c,\Phi)}_j)^{1/j}\rightarrow\infty$ as $j\rightarrow\infty$ {\itshape for each} $c>0$ are desirable (standard) properties in the theory of ultradifferentiable (and ultraholomorphic) classes, statements $(ii)$ and $(iv)$ in Proposition \ref{Phigrowthprop} suggest  that for applications the choices for $\Phi$ and $M$ should not considered to be independent; cf. \eqref{PPTsequence}.
		
		In concrete applications the requirement $(M^{(c,\Phi)}_j)^{1/j}\rightarrow\infty$ as $j\rightarrow\infty$ for each $c>0$ might be checked easily. However, even if  $M$ is log-convex, $\Phi$ is convex and $M$ and $\Phi$ are well related as in \eqref{PPTsequence}, in general as $c\rightarrow 0$ one can only expect that condition \eqref{McPhilogconv} will be satisfied from some $j_c\in\NN_{>0}$ on (and $j_c\rightarrow+\infty$ as $c\rightarrow 0$). Nevertheless, in this situation one can replace each $M^{(c,\Phi)}$ (for $c<1$ small) by some equivalent sequence when changing $M^{(c,\Phi)}$ at the beginning, i.e., only for finitely many $j$. This technical modification leaves the classes $\mathcal{E}_{\{\mathcal{M}_{M,\Phi}\}}$ and $\mathcal{E}_{(\mathcal{M}_{M,\Phi})}$ unchanged.
	\end{remark}
	
	\subsection{Comparison results}
	
	This section is devoted to formulate and prove the main comparison theorems. Using the preparation from the previous section we are in position to prove the first statement.
	
	\begin{theorem}\label{PPTasweightmatrixthm}
		Let $M\in\RR_{>0}^{\NN}$ be given and let $\Phi$ be an exponent sequence. Let
		$\mathcal{M}_{M,\Phi}$ be the matrix defined in \eqref{convenientmatrix}, then as locally convex vector spaces we get
		\begin{equation}\label{PPTasweightmatrixthmequ}
			\mathcal{E}_{\{M,\Phi\}}=\mathcal{E}_{\{\mathcal{M}_{M,\Phi}\}},\hspace{15pt}\mathcal{E}_{(M,\Phi)}=\mathcal{E}_{(\mathcal{M}_{M,\Phi})}.
		\end{equation}
		By the analogous definitions of the spaces we see that \eqref{PPTasweightmatrixthmequ} also holds for the other classes mentioned in Remark \ref{otherclassesremark}.
	\end{theorem}
	
	Both cases from Example \ref{PPTexamples} satisfy \eqref{Philiminf}; if $\Phi_j=j$ and so we are treating the classical situation, then the above result becomes trivial in the sense that $M^{(c,\Phi)}\hyperlink{approx}{\approx}M$ for all $c>0$, i.e., $\mathcal{M}_{M,\Phi}$ is constant.
	
	\demo{Proof} {\itshape The Roumieu case.}
	By definition, we have the following estimate
	$$\forall\;c\ge 1\;\forall\;h\ge 1\;\forall\;j\in\NN:\;\;\; c^{\Phi_j}M_j=M^{(c,\Phi)}_j\le h^jM^{(c,\Phi)}_j,$$
	that is verifying $\mathcal{E}_{\{M,\Phi\}}\subseteq\mathcal{E}_{\{\mathcal{M}_{M,\Phi}\}}$ (with continuous inclusion).
	
	Conversely, let $h,c\ge 1$ be given. By  Lemma \ref{lemma.MLCondition}, there exist $c_1, A \ge 1$ such that for all $j\in\NN$
	$$h^jM^{(c,\Phi)}_j \le AM^{(c_1,\Phi)}_j,$$
	which shows $\mathcal{E}_{\{\mathcal{M}_{M,\Phi}\}}\subseteq\mathcal{E}_{\{M,\Phi\}}$ (with continuous inclusion).
	\vspace{6pt}
	
	{\itshape The Beurling case.} Follows analogously, but in this case we use Lemma \ref{lemma.MLCondition} to prove the (continuous) inclusion $\mathcal{E}_{(M,\Phi)}\subseteq\mathcal{E}_{(\mathcal{M}_{M,\Phi})}$ .
	\qed\enddemo

	\vspace{5mm}
	\par
	On the other hand, let us show now that condition \eqref{Philiminf} is also necessary to obtain \eqref{PPTasweightmatrixthmequ} (or even more), when assuming mild extra assumptions on $M$.
	
	A crucial part of the proof of the Roumieu case is based upon the existence of so-called {\itshape optimal functions} in Roumieu classes: For any given normalized log-convex sequence $N$, we consider the function
	\begin{equation}\label{characterizingfct}
		\theta_N(t):=\sum_{j=0}^{\infty}\frac{N_j}{2^j\nu_j^j}\exp(2i\nu_jt),\;\;\;t\in\RR,
	\end{equation}
with $\nu_j:=\frac{N_j}{N_{j-1}}$ for $j\ge 1$ and $\nu_0:=1$. It is known that
	$$\theta_N\in\mathcal{E}_{\{N\}}(\RR,\CC),\hspace{25pt}|\theta_N^{(j)}(0)|\ge N_j\;\;\;\forall\;j\in\NN;$$
	see e.g. \cite[Thm. 1]{thilliez}, \cite[Lemma 2.9]{compositionpaper} and the detailed proof in \cite[Prop. 3.1.2]{diploma}. There it has been commented that $\theta_N\notin\mathcal{E}_{(N)}(\RR,\CC)$ and the proof shows that in \eqref{characterizingfct} we can replace $\nu_j$ by $\nu_{j+1}$.
	
	The proof of the  Beurling case makes use of the following functional analytic result.
	
	\begin{proposition}
		\label{funcanaprop}
		Let $E,F$ be Fr\'echet spaces, such that $E$ is a linear subspace of $F$ (not assuming continuous inclusion). Assume that both are continuously included in $C(U)$ (or even in any Hausdorff space), where $U$ is some open subset of $\mathbb{R}^d$. Then $E$ is continuously included in $F$.
	\end{proposition}
	
	\begin{proof}
		We want to show that the inclusion $\iota:E \rightarrow F$ is continuous. By the closed graph theorem, it suffices to show that if $f_n \rightarrow 0$ in $E$ (and thus in $C(U)$), and $f_n \rightarrow g$ in $F$ (and thus in $C(U)$), we have $g=0$. But this is now clear since $C(U)$ is Hausdorff.
	\end{proof}
	
	\begin{remark}
		In the light of the previous Proposition, any inclusion (as sets) of Beurling classes  is automatically a continuous inclusion. And equality as sets yields equality as Fr\'echet spaces.
	\end{remark}
	
	Now we are in the position to formulate a converse to Theorem \ref{PPTasweightmatrixthm}.
	
	\begin{theorem}\label{Phinecprop1}
		Let $M\in\RR_{>0}^{\NN}$ and $\Phi\in\RR_{\ge0}^{\NN}$ be given. Then we get:
		\begin{itemize}
			\item[$(i)$] The Roumieu case. Assume that $M^{(c,\Phi)}$ is log-convex and normalized for some $c \ge 1$, and let $\mathcal{L}=\{L^{(x)}: x>0\}$ be a \hyperlink{Msc}{$(\mathcal{M}_{\on{sc}})$} matrix or even only consisting of normalized log-convex weight sequences. Assume that
			$$\mathcal{E}_{\{M,\Phi\}}(\RR)=\mathcal{E}_{\{\mathcal{L}\}}(\RR)$$
			is valid (as sets). Then $\Phi$ has to satisfy \eqref{Philiminf}, i.e., $\Phi$ is an exponent sequence.
			
			\item[$(ii)$] The Beurling case. Assume that $\lim_{j\rightarrow+\infty}(M^{(c,\Phi)}_j)^{1/j}=+\infty$ for all $c>0$ and such that (w.l.o.g., cf. Remark \ref{log-remark}) each $M^{(c,\Phi)}$ is log-convex and $M^{(c,\Phi)}_0=1$. Assume also that (as sets and thus automatically as Fr\'echet spaces)
			$$\mathcal{E}_{(M,\Phi)}(\RR)=\mathcal{E}_{(\mathcal{L})}(\RR)$$
			is valid with $\mathcal{L}=\{L^{(x)}: x>0\}$ a given \hyperlink{Msc}{$(\mathcal{M}_{\on{sc}})$} matrix. Then $\Phi$ has to satisfy \eqref{Philiminf}.
		\end{itemize}
		
	\end{theorem}
	
	We immediately get the following consequence:
	
	\begin{corollary}
		Let $M\in\RR_{>0}^{\NN}$ and $\Phi\in\RR_{\ge0}^{\NN}$ be given such that the requirements of the particular case in Theorem \ref{Phinecprop1} are valid.  If $\lim_{j \rightarrow \infty} \frac{\Phi_j}{j}=0$, then $\mathcal{E}_{[M,\Phi]}$ \textit{cannot} be identified with any weight matrix class $\mathcal{E}_{[\mathcal{L}]}$ with $\mathcal{L}$ being \hyperlink{Msc}{$(\mathcal{M}_{\on{sc}})$}.
	\end{corollary}
	
	\begin{proof}\textit{(i):}\par
		Choose $c$ such that $M^{(c,\Phi)}$ is log-convex and normalized. By applying the optimal functions $\theta_N$ from \eqref{characterizingfct} to $N\equiv M^{(c,\Phi)}$ we get by the equality of the classes that there exist $h,x_0$ (w.l.o.g. greater than $1$) such that
		\begin{equation}\label{Phinecprop1equ0}
			\forall x \ge x_0~\forall j \in \NN: \quad M_j \le c^{\Phi_j}M_j\le h^{j+1}L^{(x)}_j,
		\end{equation}
		and therefore we get
		\begin{equation}
			\label{Phinecprop1equ}
			\forall x \ge x_0~\forall j \in \NN:\quad \frac{1}{h^2} \le  \left(\frac{L^{(x)}_j}{M_j}\right)^{1/j}.
		\end{equation}
		
		In addition we infer, again by working with optimal functions, but now for the sequences $j \mapsto n^jL^{(n)}_j$, that for all $n \in \NN$ there exists $c_n$ (w.l.o.g. increasing in $n)$ such that
		\begin{equation*}
			\forall\;j \in \NN:\quad n^j L^{(n)}_j \le c_n^{\Phi_j} M_j,
		\end{equation*}
		and by taking roots we end up with
		\begin{equation}
			\label{Phinecprop2equ}
			\forall\;j \in \NN_{>0}:\quad n(L^{(n)}_j)^{1/j} \le c_n^{\Phi_j/j}M_j^{1/j}.
		\end{equation}
		Now let us assume that \eqref{Philiminf} is violated, i.e., that
		\[
		\liminf_{j \rightarrow \infty} \frac{\Phi_j}{j} = 0,
		\]
		then we can find a sequence of integers $j_n$ such that
		\[
		c_n^{\Phi_{j_n}/j_n}\le 2.
		\]
		Combining \eqref{Phinecprop1equ} and \eqref{Phinecprop2equ} we infer (for all $n\in\NN$ with $n\ge x_0$)
		\[
		\frac{n}{h^2} \le n \left(\frac{L^{(n)}_{j_n}}{M_{j_n}}\right)^{1/j_n}\le 2,
		\]
		which yields a contradiction as $n \rightarrow \infty$.
		\par
		\textit{(ii):}\par
		By Proposition \ref{funcanaprop}, we infer that the spaces are isomorphic as Fr\'echet spaces. Thus we get that for any compact set $K \subset\subset \RR$ there exist $h, x_0>0$ and a compact set $J \subset \subset \RR$ such that for all $f \in \mathcal{E}_{(M,\Phi)}(\RR)=\mathcal{E}_{(\mathcal{L})}(\RR)$ we have
		\[
		\|f\|_{M,\Phi,K,1}\le \frac{1}{h}\|f\|_{L^{(x_0)},J,h},
		\]
		which yields, by plugging in the family of functions $f_s(t):=e^{ist}$,
		\[
		\exp(\omega_{M}(s))\le \frac{1}{h}\exp(\omega_{L^{(x_0)}}(s/h)).
		\]
		Due to log-convexity we can apply \eqref{Prop32Komatsu} and get from this estimate (since the sequences of $\mathcal{L}$ are pointwise ordered), that for all $x\leq x_0$ and $j\in\NN$
		\begin{equation*}
			h^{j+1} L^{(x)}_j \le M_j
		\end{equation*}
		and finally, since w.l.o.g. $h\le 1$, we get that for all $x \leq x_0$ and $j\in\NN_{>0}$
		\begin{equation}
			\label{eq:uniboundbelow}
			h^2\le \left( \frac{M_j}{L^{(x)}_j} \right)^{1/j}.
		\end{equation}
		Analogously we argue to get that for all $n \in \NN$ there exists $c_n>0$ such that
		\begin{equation*}
			c_n^{\Phi_j+1}M_j \le \left(\frac{1}{n}\right)^j L^{(1/n)}_j,
		\end{equation*}
		and therefore
		\begin{equation}
			\label{eq:uniboundabove}
			\left(\frac{M_j}{L_j^{(1/n)}}\right)^{1/j}\le \frac{1}{n} \left( \frac{1}{c_n}\right)^{(\Phi_j+1)/j}.
		\end{equation}
		Now again assume that \eqref{Philiminf} is violated, i.e., that
		\[
		\liminf_{j \rightarrow \infty} \frac{\Phi_j}{j} = 0,
		\]
		then we can find a sequence of integers $j_n$ such that
		\[
		\left(\frac{1}{c_n}\right)^{(\Phi_{j_n}+1)/j_n}\le 2.
		\]
		Combining \eqref{eq:uniboundbelow} and \eqref{eq:uniboundabove} we infer
		\[
		h^2\le \frac{2}{n},
		\]
		which again gives the desired contradiction and thus finishes the proof.
	\end{proof}
	
	In particular, if we choose for the matrix $\mathcal{L}$ the concrete matrix $\mathcal{M}_{M,\Phi}$ from \eqref{convenientmatrix}, then we can draw the same conclusion i.e., that $\Phi$ already has to be an exponent sequence. Under somewhat milder conditions, we can actually show even more in this case. This is due to the fact that we can prove the desired implication directly, however by using the same techniques (\eqref{characterizingfct}, Proposition \ref{funcanaprop}) as in the proof of Theorem \ref{Phinecprop1} before.

	\begin{theorem}\label{Phinecprop}
		Let $M\in\RR_{>0}^{\NN}$ and $\Phi\in\RR_{\ge0}^{\NN}$ be given and
		$\mathcal{M}_{M,\Phi}$ be the matrix defined in \eqref{convenientmatrix}. Then we get:
		\begin{itemize}
			\item[$(i)$] The Roumieu case. Assume that $M^{(c,\Phi)}$ is log-convex and normalized for some $c>0$ and that (as sets)
			$$\mathcal{E}_{\{\mathcal{M}_{M,\Phi}\}}\subseteq\mathcal{E}_{\{M,\Phi\}}$$
			is valid. Then $\Phi$ has to satisfy \eqref{Philiminf}. In particular this implication holds for any $M\in\hyperlink{LCset}{\mathcal{LC}}$.
			
			\item[$(ii)$] The Beurling case. Assume that $\lim_{j\rightarrow+\infty}(M^{(c,\Phi)}_j)^{1/j}=+\infty$ for all $c>0$ and such that (w.l.o.g., cf. Remark \ref{log-remark})  each $M^{(c,\Phi)}$ is log-convex and  $M^{(c,\Phi)}_0=1$. Assume also that (as sets and thus automatically as Fr\'echet spaces)
			$$\mathcal{E}_{(M,\Phi)}\subseteq\mathcal{E}_{(\mathcal{M}_{M,\Phi})}$$
			is valid. Then $\Phi$ has to satisfy \eqref{Philiminf}.
		\end{itemize}
	\end{theorem}
	
	Consequently, \eqref{Philiminf} has to hold when assuming \eqref{PPTasweightmatrixthmequ} if the additional requirements on $M^{(c,\Phi)}$ of the particular case hold true.\vspace{6pt}
	
	If the symmetric restriction from above is imposed on the sequence $\Phi$, i.e., the growth of $h^{\Phi_j}$ is at most geometric, then we recover the classical ultradifferentiable classes
	defined by a single weight sequence.
	
	\begin{proposition}\label{limsupprop}
		Let $M\in\RR_{>0}^{\NN}$ and $\Phi=(\Phi_j)_j\in\RR_{\ge 0}^{\NN}$ be given. If $\Phi$ is an exponent sequence and in addition also
		\begin{equation}\label{Philimsup}
			\limsup_{j \rightarrow\infty} \frac{\Phi_j}{j}<\infty,
		\end{equation}
		then as locally convex vector spaces $\mathcal{E}_{[M,\Phi]}=\mathcal{E}_{[M]}$.
	\end{proposition}
	
	Finally, we can treat the converse statement.
	
	\begin{theorem}\label{limsuppropconv}
		Let $M\in\RR_{>0}^{\NN}$ and $\Phi\in\RR_{\ge0}^{\NN}$ be given. Then we get:
		\begin{itemize}
			\item[$(i)$] The Roumieu case. Assume that $M^{(c,\Phi)}$ is log-convex and normalized for some $c>1$, and let $M\in\hyperlink{LCset}{\mathcal{LC}}$ or even $M$ be only normalized and log-convex. Assume that
			$$\mathcal{E}_{\{M,\Phi\}}(\RR)=\mathcal{E}_{\{M\}}(\RR)$$
			is valid (as sets). Then $\Phi$ has to satisfy both \eqref{Philiminf} and \eqref{Philimsup}.
			
			\item[$(ii)$] The Beurling case. Assume that $\lim_{j\rightarrow+\infty}(M^{(c,\Phi)}_j)^{1/j}=+\infty$ for all $c>0$ and such that (w.l.o.g., cf. Remark \ref{log-remark}) each $M^{(c,\Phi)}$ is log-convex and $M^{(c,\Phi)}_0=1$. Assume also that (as sets/locally convex vector spaces)
			$$\mathcal{E}_{(M,\Phi)}(\RR)=\mathcal{E}_{(M)}(\RR)$$
			is valid. Then $\Phi$ has to satisfy both \eqref{Philiminf} and \eqref{Philimsup}.
		\end{itemize}
		
	\end{theorem}
	
	\begin{proof}
		We follow the proof of Theorem \ref{Phinecprop1} with $M=L^{(x)}$ for any $x>0$. In the Roumieu case the second part in \eqref{Phinecprop1equ0} implies \eqref{Philimsup} for $\Phi$. Then we consider $j\mapsto h^jM_j$ for some arbitrary but fixed $h>1$ instead of $j\mapsto n^jL^{(n)}_j$ and so \eqref{Phinecprop2equ} yields \eqref{Philiminf}.\vspace{6pt}
		
		In the Beurling case note that $M^{(1,\Phi)}=M$ and by the assumption $M$ has all assumptions from the set \hyperlink{LCset}{$\mathcal{LC}$} except $M_0\le M_1$. Then follow again the proof of Theorem \ref{Phinecprop1}: Replace $M$ by $(c^{\Phi_j}M_j)_j$ for some $1>c>0$ and $L^{(x)}=M$ and get with $h<1$ in the first step that \eqref{Philimsup} holds. The argument for obtaining \eqref{Philiminf} follows similarly.
	\end{proof}

	\subsection{More general (matrix) situations}

	For the sake of completeness let us comment on even more general situations compared with the definitions in Section \ref{ultrabeyondsection}. On the one hand, we can consider $\Phi$-ultradifferentiable classes $\mathcal{E}_{[\mathcal{N},\Phi]}$ defined in terms of a given weight matrix $\mathcal{N}=\{N^{(x)}: x>0\}$, i.e., with an additional parameter.
	 For $U\subseteq\RR^d$ non-empty open, these classes are defined in the natural way by
	\begin{equation*}
		\mathcal{E}_{\{\mathcal{N},\Phi\}}(U)=\underset{K\subset\subset U}{\varprojlim}\;\underset{x\in\mathcal{I}}{\varinjlim}\;\underset{h>0}{\varinjlim}\;\mathcal{E}_{N^{(x)},\Phi,h}(K).
	\end{equation*}
	Similarly, we consider for the Beurling case
	\begin{equation*}
		\mathcal{E}_{(\mathcal{N},\Phi)}(U)=\underset{K\subset\subset U}{\varprojlim}\;\underset{x\in\mathcal{I}}{\varprojlim}\;\underset{h>0}{\varprojlim}\;\mathcal{E}_{N^{(x)},\Phi,h}(K).
	\end{equation*}
    Accordingly, we introduce (cf. \eqref{convenientmatrix})
	\begin{equation}\label{convenientmatrixmatrix}
		\mathcal{M}_{\mathcal{N},\Phi}:=\{N^{(c,c,\Phi)}: c>0\},\hspace{15pt}N^{(c,c,\Phi)}_j:=c^{\Phi_j}N^{(c)}_j,\;\;\;j\in\NN.
	\end{equation}
	It is then straight-forward to check that Lemma \ref{lemma.MLCondition} can be transferred to this setting and Theorem \ref{PPTasweightmatrixthm} takes the following form:
	
	\begin{theorem}\label{PPTasweightmatrixthmmatrix}
		Let $\mathcal{N}=\{N^{(x)}: x>0\}$ be given and let $\Phi$ be an exponent sequence. Let
		$\mathcal{M}_{\mathcal{N},\Phi}$ be the matrix defined in \eqref{convenientmatrixmatrix}, then as locally convex vector spaces we get
		\begin{equation}\label{PPTasweightmatrixthmmatrixequ}
			\mathcal{E}_{\{\mathcal{N},\Phi\}}=\mathcal{E}_{\{\mathcal{M}_{\mathcal{N},\Phi}\}},\hspace{15pt}\mathcal{E}_{(\mathcal{N},\Phi)}=\mathcal{E}_{(\mathcal{M}_{\mathcal{N},\Phi})}.
		\end{equation}

In both cases we can replace the symbol (functor) $\mathcal{E}$ by $\mathcal{B}$, $\mathcal{D}$, $\mathcal{A}$ or by $\Lambda$.
	\end{theorem}
	
		 These classes will be relevant for the study of PTT-limit classes in Section \ref{PTTnonfixedparametersect}. Theorem \ref{PPTasweightmatrixthmmatrix} and the matrix introduced in \eqref{convenientmatrixmatrix} should be compared with the matrix $\mathcal{M}^{\sigma}$, see \eqref{weightmatrixsigma}; in particular, this result becomes relevant for the equalities in Remark \ref{alternativedefrem}.\vspace{12pt}
	
	On the other hand, take $M\in\RR_{>0}^{\NN}$ and let $\mathcal{F}:=\{\Phi^a: a>0\}$ be a {\itshape family of sequences} $\Phi^a\in\RR_{\ge 0}^{\NN}$ such that
	\begin{equation}\label{expomatrixrequire}
		\forall\;0<a\le b\;\forall\;j\in\NN:\;\;\;\Phi^a_j\log(a)\le\Phi^b_j\log(b).
	\end{equation}
	
	We introduce the following locally convex vector spaces
	$$\mathcal{E}_{\{M,\mathcal{F}\}}(K):=\underset{a>0}{\varinjlim}\;\underset{h>0}{\varinjlim}\;\mathcal{E}_{M,\Phi^a,h}(K),$$
	and
	$$\mathcal{E}_{\{M,\mathcal{F}\}}(U)=\underset{K\subset\subset U}{\varprojlim}\;\underset{a>0}{\varinjlim}\;\underset{h>0}{\varinjlim}\;\mathcal{E}_{M,\Phi^a,h}(K)=\underset{K\subset\subset U}{\varprojlim}\;\mathcal{E}_{\{M,\mathcal{F}\}}(K).$$
	Similarly, we set
	$$\mathcal{E}_{(M,\mathcal{F})}(K):=\underset{a>0}{\varprojlim}\;\underset{h>0}{\varprojlim}\;\mathcal{E}_{M,\Phi^{a},h}(K),$$
	and
	$$\mathcal{E}_{(M,\mathcal{F})}(U)=\underset{K\subset\subset U}{\varprojlim}\;\underset{a>0}{\varprojlim}\;\underset{h>0}{\varprojlim}\;\mathcal{E}_{M,\Phi^{a},h}(K)=\underset{K\subset\subset U}{\varprojlim}\;\mathcal{E}_{(M,\Phi^{a})}(K).$$
	
	Finally, let us introduce the matrix
	\begin{equation}\label{convenientmatrixF}
		\mathcal{M}_{M,\mathcal{F}}:=\{M^{(c,\Phi^c)}: c>0\},\hspace{25pt}M^{(c,\Phi^c)}_j:=c^{\Phi^c_j}M_j,\;\;\;j\in\NN.
	\end{equation}
	
	If $\Phi^a=\Phi$ for all $a>0$, then \eqref{expomatrixrequire} is trivially satisfied (recall that $\Phi_j\ge 0$) and $\mathcal{E}_{[M,\mathcal{F}]}=\mathcal{E}_{[M,\Phi]}$ as locally convex vector spaces.
	
	Theorem \ref{PPTasweightmatrixthm} turns in the following form:
	
	\begin{theorem}\label{PPTasweightmatrixthmF}
		Let $M\in\RR_{>0}^{\NN}$ be given and let $\mathcal{F}:=\{\Phi^a: a>0\}$ be a family of sequences $\Phi^a\in\RR_{\ge 0}^{\NN}$ satisfying \eqref{expomatrixrequire}. Let $\mathcal{M}_{M,\mathcal{F}}$ be the matrix defined in \eqref{convenientmatrixF} and assume that this matrix satisfies
		\hyperlink{R-L}{$(\mathcal{M}_{\{\on{L}\}})$} resp. \hyperlink{B-L}{$(\mathcal{M}_{(\on{L})})$}.
		Then as locally convex vector spaces we get
		\begin{equation}\label{PPTasweightmatrixthmFequ}
			\mathcal{E}_{\{M,\mathcal{F}\}}=\mathcal{E}_{\{\mathcal{M}_{M,\mathcal{F}}\}},\hspace{20pt}\mathcal{E}_{(M,\mathcal{F})}=\mathcal{E}_{(\mathcal{M}_{M,\mathcal{F}})}.
		\end{equation}
		Again, in both cases we can replace the symbol (functor) $\mathcal{E}$ by $\mathcal{B}$, $\mathcal{D}$, $\mathcal{A}$ or by $\Lambda$.
	\end{theorem}
	
	\demo{Proof} Analogous to Theorem \ref{PPTasweightmatrixthm}.\qed\enddemo

	Let us characterize now the crucial conditions \hyperlink{R-L}{$(\mathcal{M}_{\{\on{L}\}})$} resp. \hyperlink{B-L}{$(\mathcal{M}_{(\on{L})})$} in terms of a growth condition on $\mathcal{F}$. The next result generalizes Lemma \ref{lemma.MLCondition} to the matrix $\mathcal{M}_{M,\mathcal{F}}$ defined in \eqref{convenientmatrixF}.
	
	\begin{proposition}
		Let $\mathcal{M}_{M,\mathcal{F}}$ be given and assume that $\mathcal{F}=\{\Phi^a: a>0\}$ satisfies \eqref{expomatrixrequire}.
		\begin{itemize}
			\item[$(a)$] The following are equivalent (Roumieu case):
			\begin{itemize}
				\item[$(i)$] $\mathcal{M}_{M,\mathcal{F}}:=\{M^{(c,\Phi^c)}: c>0\}$ satisfies \hyperlink{R-L}{$(\mathcal{M}_{\{\on{L}\}})$}.
				
				\item[$(ii)$] The family $\mathcal{F}$ satisfies
				\begin{equation}\label{familyFabsorbingroum}
					\exists\;\epsilon>0\;\forall\;c>0\;\exists\;d>c:\;\;\;\liminf_{j\rightarrow\infty}\frac{\Phi_j^d}{j}\log(d)-\frac{\Phi_j^c}{j}\log(c)\ge\epsilon.
				\end{equation}
			\end{itemize}
			
			\item[$(b)$] The following are equivalent (Beurling case):
			\begin{itemize}
				\item[$(i)$] $\mathcal{M}_{M,\mathcal{F}}:=\{M^{(c,\Phi^c)}: c>0\}$ satisfies \hyperlink{B-L}{$(\mathcal{M}_{(\on{L})})$}.
				
				\item[$(ii)$] The family $\mathcal{F}$ satisfies
				\begin{equation}\label{familyFabsorbingbeur}
					\exists\;\epsilon>0\;\forall\;c>0\;\exists\;d<c:\;\;\;\liminf_{j\rightarrow\infty}\frac{\Phi_j^c}{j}\log(c)-\frac{\Phi_j^d}{j}\log(d)\ge\epsilon.
				\end{equation}
			\end{itemize}
		\end{itemize}
	\end{proposition}
	
	\demo{Proof}
	
	$(a)(i)\Rightarrow(ii)$ By assumption we have (recall $M^{(c,\Phi^c)}_j:=c^{\Phi^c_j}M_j$):
	$$\forall\;h>0\;\forall\;c>0\;\exists\;d>0\;\exists\;D\ge 1\;\forall\;j\in\NN:\;\;\;h^jc^{\Phi^c_j}M_j\le Dd^{\Phi^d_j}M_j.$$
	Fix now $h>1$ and by \eqref{expomatrixrequire} we can assume that $d\ge c$. Hence
	$$\forall\;c>0\;\exists\;d>c\;\exists\;D\ge 1\;\forall\;j\in\NN_{>0}:\;\;\;\log(h)-\frac{\log(D)}{j}\le\frac{\Phi_j^d}{j}\log(d)-\frac{\Phi_j^c}{j}\log(c),$$
	and so \eqref{familyFabsorbingroum} is verified with (e.g.) $\epsilon:=\log(h)/2$.
	
	$(a)(ii)\Rightarrow(i)$ \eqref{familyFabsorbingroum} implies
	$$\exists\;\epsilon>0\;\forall\;c>0\;\exists\;d>c\;\exists\;j_c\in\NN\;\forall\;j\ge j_c:\;\;\;\frac{\Phi_j^d}{j}\log(d)-\frac{\Phi_j^c}{j}\log(c)\ge\frac{\epsilon}{2}\Leftrightarrow e^{\epsilon j/2}c^{\Phi^c_j}M_j\le d^{\Phi^d_j}M_j.$$
	Then let $h>1$ be given (large) and iterate the previous estimate $n$-times, with $n\in\NN_{>0}$ chosen minimal such that $e^{n\epsilon/2}\ge h$. This then yields choices $d=c_{n+1}>c_n>\dots > c_1=c$ (since by assumption the value of $\epsilon$ is not depending on the choice for $c_i$) such that
	$$h^jc^{\Phi^c_j}M_j\le e^{n\epsilon j/2}c^{\Phi^c_j}M_j\le d^{\Phi^{d}_j}M_j$$
	for all $j\ge\max\{j_{c_i}: 1\le i\le n\}$. Finally, when choosing $D\ge 1$ sufficiently large, we ensure
	$$h^jc^{\Phi^c_j}M_j\le D d^{\Phi^{d}_j}M_j$$
	for all $j\in\NN$ and $D$ is only depending on the number of iterations $n$, i.e., on given $h$, and on the given index $c$. Thus \hyperlink{R-L}{$(\mathcal{M}_{\{\on{L}\}})$} is verified.\vspace{6pt}
	
	The equivalence for the Beurling case is analogous.
	\qed\enddemo
	
	\begin{remark}
		We comment on some special cases:
		\begin{itemize}
			\item[$(a)$] The constant case: If $\Phi^a=\Phi$ for all $a>0$ and if $\Phi$ is an exponent sequence, i.e., \eqref{Philiminf} is valid, then both \eqref{familyFabsorbingroum} and \eqref{familyFabsorbingbeur} hold true: For given $c>0$ e.g. we choose $d=2c$ in \eqref{familyFabsorbingroum} resp. $d=c/2$ in \eqref{familyFabsorbingbeur} and get both requirements with $\epsilon:=\log(2)\liminf_{j\rightarrow\infty}\frac{\Phi_j}{j}$.
			
			\item[$(b)$] Assume that for all $c,d>0$ with $c\leq d$ we have that $\Phi^c \leq \Phi^d$, which implies \eqref{expomatrixrequire}.
			
			\begin{itemize}
				\item[$(*)$] Assume that there exists some $c_0>0$ such that $\Phi^{c_0}$ satisfies \eqref{Philiminf} with value $\epsilon_0>0$. So each $\Phi^d$, $d\ge c_0$, satisfies \eqref{Philiminf} with $\liminf_{j\rightarrow\infty}\frac{\Phi^d_j}{j}\ge \epsilon_0$. Then, arguing as in the constant case before, we get \eqref{familyFabsorbingroum} with $\epsilon:=\log(2)\epsilon_0$ for all choices $c\ge c_0$. Note that in the Roumieu case we can omit all $c<c_0$ without changing the corresponding function class.
				
				\item[$(*)$] If for all $c>0$ we have that $\Phi^c$ satisfies \eqref{Philiminf} {\itshape uniformly} in $c$, i.e.,
				$$\exists\;\epsilon_1>0\;\forall\;c>0:\;\;\;\liminf_{j\rightarrow\infty}\frac{\Phi^c_j}{j}\ge\epsilon_1,$$
				then \eqref{familyFabsorbingbeur} holds true with $\epsilon:=\log(2)\epsilon_1$.
			\end{itemize}
		\end{itemize}
	\end{remark}

	
		\section{PTT-classes as spaces defined by weight matrices}\label{PTTsection}

	Let the parameters $\tau>0$ and $\sigma>1$ be given but from now on fixed and consider (with the convention $0^0:=1$)
	\begin{equation}\label{PPtsequences}
		M_j=M^{\tau,\sigma}_j:=j^{\tau j^{\sigma}},\hspace{15pt}\Phi_j=j^{\sigma} \qquad \text{for all} \quad  j\in\NN.
	\end{equation}
	 For these particular choices of $M$ and $\Phi$ we write $M^{(c,\tau,\sigma)}$ for $M^{(c,\Phi)}$. Thus
	the sequences and the matrix introduced in \eqref{convenientmatrix} have the form
	\begin{equation}\label{PPTconvenientmatrix}
		M^{(c,\tau,\sigma)}_j:=M^{(c,\Phi)}_j= c^{j^{\sigma}}j^{\tau j^{\sigma}},\;\;\;c>0,\;j\in\NN, \qquad \mathcal{M}^{\tau,\sigma} := \mathcal{M}_{M,\Phi}=\{ M^{(c,\tau,\sigma)} \,\,\,:\;\;\;c>0 \}.
	\end{equation}
	
	\subsection{Properties of the matrix $\mathcal{M}^{\tau,\sigma}$}\label{PTTfixedparameter}
	
	Note that, in particular, Theorem \ref{PPTasweightmatrixthm} applies to this special situation. We thus have as a corollary, in accordance with the notation in the works of S. Pilipovi\'{c}, N. Teofanov, and F. Tomi\'{c}, the following statement.

	\begin{proposition}\label{PPTasweightmatrixproposition}		Let $U \subseteq \RR^d$ be open, $\tau>0$, and $\sigma >1$. Then (as locally convex vector spaces)
		\begin{equation}
			\label{eq:pttdef}
			\mathcal{E}_{\{\tau,\sigma\}}(U)=\mathcal{E}_{\{\mathcal{M}^{\tau,\sigma}\}}(U),\qquad \mathcal{E}_{(\tau,\sigma)}(U)= \mathcal{E}_{(\mathcal{M}^{\tau,\sigma})}(U).
		\end{equation}
	\end{proposition}
	
	Therefore, we may apply certain results available in the weight matrix setting to PTT-classes. First we need to study the properties of the defining weight matrix $\mathcal{M}^{\tau,\sigma}$.
	
	\begin{theorem}\label{fixedparammainresult}
		Let $\tau>0$, and $\sigma >1$ be fixed. Then we have:
		\begin{itemize}
			\item[$(i)$] $\mathcal{M}^{\tau,\sigma}$ satisfies \hyperlink{B-Comega}{$(\mathcal{M}_{(\on{C}^{\omega})})$}; in fact we even have $\left(M^{(c,\tau,\sigma)}_j/j!^{\alpha}\right)^{1/j}\rightarrow+\infty$ as $j\rightarrow+\infty$ for any $\alpha>0$ and any $c>0$.  Consequently, $\mathcal{M}^{\tau,\sigma}$ also satisfies  \hyperlink{holom}{$(\mathcal{M}_{\mathcal{H}})$} and \hyperlink{R-Comega}{$(\mathcal{M}_{\{\on{C}^{\omega}\}})$}.
			
			\item[$(ii)$] There exists a matrix $\widetilde{\mathcal{M}}^{\tau,\sigma}$ which is equivalent to $\mathcal{M}^{\tau,\sigma}$ and such that $\widetilde{\mathcal{M}}^{\tau,\sigma}$ consists only of sequences that are strongly log-convex (and normalized).
			
			\item[$(iii)$] $\mathcal{M}^{\tau,\sigma}$ has both \hyperlink{R-dc}{$(\mathcal{M}_{\{\on{dc}\}})$} and \hyperlink{B-dc}{$(\mathcal{M}_{(\on{dc})})$}.

			\item[$(iv)$] $\mathcal{M}^{\tau,\sigma}$ has both \hyperlink{R-rai}{$(\mathcal{M}_{\{\on{rai}\}})$} and \hyperlink{B-rai}{$(\mathcal{M}_{(\on{rai})})$}; in fact in both conditions we can choose the same index $\alpha=\beta$.
			
			\item[$(v)$] $\mathcal{M}^{\tau,\sigma}$ has both \hyperlink{R-FdB}{$(\mathcal{M}_{\{\on{FdB}\}})$} and \hyperlink{B-FdB}{$(\mathcal{M}_{(\on{FdB})})$}.
			
			\item[$(vi)$] For each $c>0$ the sequence $M^{(c,\tau,\sigma)}$
			is strongly non-quasianalytic, in fact we even have that $\gamma(M^{(c,\tau,\sigma)})=+\infty$. For the precise definition, properties and meanings of the growth index $\gamma(M)$ introduced in \cite[Sect. 1.3]{Thilliezdivision} we refer to \cite[Sect. 3]{index}.
			
			\item[$(vii)$] $\mathcal{M}^{\tau,\sigma}$ neither has \hyperlink{R-mg}{$(\mathcal{M}_{\{\on{mg}\}})$} nor \hyperlink{B-mg}{$(\mathcal{M}_{(\on{mg})})$}.
			
			\item[$(viii)$] The sequences $M^{(c,\tau,\sigma)}$ are pairwise non-equivalent. More precisely, we have $M^{(c_1,\tau,\sigma)}\hyperlink{triangle}{\vartriangleleft}M^{(c_2,\tau,\sigma)}$ for all $0<c_1<c_2$.
		\end{itemize}
	\end{theorem}
	
	\demo{Proof}
	$(i)$ For all $j\ge 1$ we see that
	$$\left(m^{(c,\tau,\sigma)}_j\right)^{1/j}=\left(\frac{M^{(c,\tau,\sigma)}_j}{j!}\right)^{1/j}\ge\left(\frac{M^{(c,\tau,\sigma)}_j}{j^j}\right)^{1/j}=\left(c^{j^{\sigma}}j^{\tau j^{\sigma}-j}\right)^{1/j}=c^{j^{\sigma-1}}j^{\tau j^{\sigma-1}-1},$$
	and so $\left(m^{(c,\tau,\sigma)}_j\right)^{1/j}\rightarrow+\infty$ as $j\rightarrow\infty$ for all $c>0$. This also implies $\mu^{(c,\tau,\sigma)}_j/j\rightarrow+\infty$ for all $c>0$; see \eqref{mucompare}. More generally, for any $\alpha>0$ and any $c>0$ it is immediate by the same estimate above that $\left(M^{(c,\tau,\sigma)}_j/j!^{\alpha}\right)^{1/j}\rightarrow+\infty$.\vspace{6pt}
	
	$(ii)$ Obviously, by the convention $0^0:=1$, we get $1=M^{(c,\tau,\sigma)}_0$ and $c=M^{(c,\tau,\sigma)}_1$ for all $c>0$. Thus for each $c\ge 1$ the sequence $M^{(c,\tau,\sigma)}$ is log-convex because $M$ is log-convex and $j \mapsto j^\sigma$ is convex. Actually $M^{(c,\tau,\sigma)}\in\hyperlink{LCset}{\mathcal{LC}}$ for each $c \ge 1$.
	
	Let us verify that the quotients of each sequence $m^{(c,\tau,\sigma)}$ are non-decreasing from some index $j_c$ on. This then implies the statement, since each $M^{(c,\tau,\sigma)}$ can then be replaced by an equivalent sequence $\widetilde{M}^{(c,\tau,\sigma)}$ which is even strongly log-convex and normalized:\par
	By $(i)$ we have $\mu^{(c,\tau,\sigma)}_j/j\rightarrow+\infty$ and so $\mu^{(c,\tau,\sigma)}_j/j\ge 1$ for all $j\ge j'_c$. Then take $j''_c:=\max\{j_c,j'_c\}$ and set
	$$\frac{\widetilde{\mu}^{(c,\tau,\sigma)}_j}{j}:=1,\;\;\;1\le j<j''_c,\hspace{15pt}\frac{\widetilde{\mu}^{(c,\tau,\sigma)}_j}{j}:=\frac{\mu^{(c,\tau,\sigma)}_j}{j},\;\;\;j\ge j''_c.$$
	Since $c\mapsto j''_c$ is non-decreasing (by the order of $\mu^{(c,\tau,\sigma)}$) we have that $\widetilde{m}^{(c,\tau,\sigma)}$ are ordered (even w.r.t. their quotient sequences). Moreover, $M^{(c,\tau,\sigma)}\hyperlink{approx}{\approx}\widetilde{M}^{(c,\tau,\sigma)}$ for each $c>0$ (even on the level of the corresponding quotient sequences).\vspace{6pt}
	
	So let us show that the quotients of $m^{(c,\tau,\sigma)}$ are eventually non-decreasing for any fixed $c>0$:\par
	$j\mapsto\frac{m^{(c,\tau,\sigma)}_j}{m^{(c,\tau,\sigma)}_{j-1}}$ is non-decreasing if and only if $j\mapsto\log\left(\frac{m^{(c,\tau,\sigma)}_j}{m^{(c,\tau,\sigma)}_{j-1}}\right)$ is so. For all $j\ge 1$ we get
	$$\frac{m^{(c,\tau,\sigma)}_j}{m^{(c,\tau,\sigma)}_{j-1}}=\mu^{(c,\tau,\sigma)}_j\frac{1}{j}=c^{j^{\sigma}-(j-1)^{\sigma}}\frac{j^{\tau j^{\sigma}}}{(j-1)^{\tau (j-1)^{\sigma}}}\frac{1}{j},$$
	and we set now $f(t):=\log(c)(t^{\sigma}-(t-1)^{\sigma})+\tau t^{\sigma}\log(t)-\tau(t-1)^{\sigma}\log(t-1)-\log(t)$, $t>1$. Then for all $t>1$:
	\begin{align*}
		f'(t)&=\sigma\log(c)(t^{\sigma-1}-(t-1)^{\sigma-1})+\sigma\tau t^{\sigma-1}\log(t)+\tau t^{\sigma}\frac{1}{t}
		\\&
		-\tau\sigma(t-1)^{\sigma-1}\log(t-1)-\tau(t-1)^{\sigma}\frac{1}{t-1}-\frac{1}{t}
		\\&
		=(\sigma\log(c)+\tau)(t^{\sigma-1}-(t-1)^{\sigma-1})+\tau\sigma(t^{\sigma-1}\log(t)-(t-1)^{\sigma-1}\log(t-1))-\frac{1}{t}\ge 0
		\\&
		\Leftrightarrow(\sigma\log(c)+\tau)t(t^{\sigma-1}-(t-1)^{\sigma-1})+\tau\sigma t(t^{\sigma-1}\log(t)-(t-1)^{\sigma-1}\log(t-1))\ge 1.
	\end{align*}
	We now continue to show
	\begin{equation}
		\label{eq:suffices}
		(\sigma\log(c)+\tau)t(t^{\sigma-1}-(t-1)^{\sigma-1})+\tau\sigma t\log(t-1)(t^{\sigma-1}-(t-1)^{\sigma-1}) \quad \rightarrow \quad \infty, ~ t \rightarrow \infty,
	\end{equation}
	which obviously implies the above statement.\par
	
	Observe that $t^{\sigma-1}-(t-1)^{\sigma-1}=(\sigma-1)\xi^{\sigma-2}$ for some $\xi\in(t-1,t)$. Therefore, for $t>2$ and $\sigma > 2$, we have
	$$
	(\sigma-1)t^{\sigma-2}\ge t^{\sigma-1}-(t-1)^{\sigma-1}\ge(\sigma-1)(t-1)^{\sigma - 2},
	$$
	whereas in the case $1<\sigma\le 2$, we have
	\[
	(\sigma-1)(t-1)^{\sigma-2}\ge t^{\sigma-1}-(t-1)^{\sigma-1} \ge (\sigma-1)t^{\sigma - 2}.
	\]
	Plugging in the appropriate term into \eqref{eq:suffices} finishes the proof.

	
	$(iii)$
	We claim that for some choices $c,c_1>0$ and some $A\ge 1$ (large) we get for all $j\in\NN_{>0}$:
	
	\begin{align*}
		&M^{(c,\tau,\sigma)}_{j+1}\le AM^{(c_1,\tau,\sigma)}_j\Leftrightarrow c^{(j+1)^{\sigma}}(j+1)^{\tau(j+1)^{\sigma}}\le A c_1^{j^{\sigma}}j^{\tau j^{\sigma}}
		\\&
		\Leftrightarrow(j+1)^{\sigma}\log(c)+\tau(j+1)^{\sigma}\log(j+1)\le\log(A)+j^{\sigma}\log(c_1)+\tau j^{\sigma}\log(j)
		\\&
		\Leftrightarrow\tau\underbrace{\frac{(j+1)^{\sigma}\log(j+1)-j^{\sigma}\log(j)}{(j+1)-j}}_{=:\Delta_j}\le\log(A)+j^{\sigma}\log(c_1)-(j+1)^{\sigma}\log(c).
	\end{align*}
	Note that the very first estimate is clear for $j=0$ by the convention $0^0:=1$ and when taking $A\ge c^{\sigma}$. Set $f(t):=t^{\sigma}\log(t)$, $t\ge 1$, and then $f'(t)=\sigma t^{\sigma-1}\log(t)+t^{\sigma-1}=t^{\sigma-1}(\sigma\log(t)+1)$ is strictly increasing (and tending to infinity as $t\rightarrow+\infty$). Thus we have $\Delta_j\le f'(j+1)=(j+1)^{\sigma-1}(\sigma\log(j+1)+1)\le 2\sigma(j+1)^{\sigma-1}\log(j+1)$. On the other hand, when given $c\ge 1$ we choose $c_1:=(2c)^{2^{\sigma}}(>c)$, and then for all $j\ge 1$:
	\begin{align*}
		j^{\sigma}\log(c_1)-(j+1)^{\sigma}\log(c)&=(2j)^{\sigma}\log(2c)-(j+1)^{\sigma}\log(c)
		\\&
		\ge(j+1)^{\sigma}\log(2c)-(j+1)^{\sigma}\log(c)=(j+1)^{\sigma}\log(2).
	\end{align*}
	Thus, in the Roumieu case we are able to conclude when choosing $A\ge 1$ sufficiently large.
	
	When given $c_1<1$, then we choose $c:=\frac{c_1}{2}$ and get for all $j\ge 1$:
	\begin{align*}
		j^{\sigma}\log(c_1)-(j+1)^{\sigma}\log(c)&=j^{\sigma}\log(c_1)-(j+1)^{\sigma}\log(c_1)+(j+1)^{\sigma}\log(2)
		\\&
		\ge (j+1)^{\sigma}\log(2),
	\end{align*}
		since $\log(c_1)<0$. This proves the Beurling case.\vspace{6pt}
	

	$(iv)$ First, by Stirling's formula we get for all $c>0$ and $j\ge 1$:
	\begin{align*}
		&c^{j^{\sigma-1}}j^{\tau j^{\sigma-1}-1}=\frac{c^{j^{\sigma-1}}j^{\tau j^{\sigma-1}}}{j}=\frac{(M^{(c,\tau,\sigma)}_j)^{1/j}}{j}
		\\&
		\le(m^{(c,\tau,\sigma)}_j)^{1/j}\le\frac{e}{j}(M^{(c,\tau,\sigma)}_j)^{1/j}=ec^{j^{\sigma-1}}j^{\tau j^{\sigma-1}-1}.
	\end{align*}
	Thus \hyperlink{R-rai}{$(\mathcal{M}_{\{\on{rai}\}})$} follows because we have for all $1\le c\le c_1$, $A\ge e$ and $1\le j\le k$:
	\begin{align*}
		ec^{j^{\sigma-1}}j^{\tau j^{\sigma-1}-1}\le Ac_1^{k^{\sigma-1}}k^{\tau k^{\sigma-1}-1}.
	\end{align*}
	Concerning \hyperlink{B-rai}{$(\mathcal{M}_{(\on{rai})})$}, let $0<c_1\le c<1$, $A\ge 1$ and $1\le j\le k$, then
	\begin{align*}
		&ec_1^{j^{\sigma-1}}j^{\tau j^{\sigma-1}-1}\le Ac^{k^{\sigma-1}}k^{\tau k^{\sigma-1}-1}
		\\&
		\Leftrightarrow j^{\sigma-1}\log(c_1)-k^{\sigma-1}\log(c)\le\log(A/e)+\left(\tau k^{\sigma-1}-1\right)\log(k)-\left(\tau j^{\sigma-1}-1\right)\log(j).
	\end{align*}
	So the desired estimate follows by choosing $c_1=c$ and $A$ large enough.\vspace{6pt}
	
	
	Recall that each strongly log-convex sequence satisfying $m_0=M_0=1$ has the property that $j\mapsto(m_j)^{1/j}$ is nondecreasing, compare this with $(ii)$.\vspace{6pt}
	
	$(v)$ This follows by $(i)$, $(iii)$ and $(iv)$; see \cite[Lemma 1 $(1)$]{characterizationstabilitypaper}.\vspace{6pt}
	
	$(vi)$ By repeating the arguments given in $(ii)$ we see that for each $\alpha>0$ (and any $c>0$) the mapping $j\mapsto\mu^{(c,\tau,\sigma)}_j\frac{1}{j^{\alpha}}$ is eventually non-decreasing and by $(i)$ one has $\left(\frac{M^{(c,\tau,\sigma)}_j}{j!^{\alpha}}\right)^{1/j}\rightarrow+\infty$ for any $\alpha>0$.
	
	Thus $\gamma(M^{(c,\tau,\sigma)})=+\infty$ follows for all $c>0$ (however this implication is in general strict; see \cite[Prop. 4.4 $(i)\Rightarrow(ii)$]{surjectivityregularsequences}) and so \cite[Thm. 3.11]{index} yields the assertion.
	
	$(vii)$ We test conditions \hyperlink{R-mg}{$(\mathcal{M}_{\{\on{mg}\}})$} and \hyperlink{B-mg}{$(\mathcal{M}_{(\on{mg})})$} for $j=k\ge 1$. So let $A\ge 1$ and $c,c_1>0$ be given, then:
	\begin{align*}
		&(M^{(c,\tau,\sigma)}_{2j})^{1/(2j)}\le A(M^{(c_1,\tau,\sigma)}_j)^{1/j}\Leftrightarrow c^{(2j)^{\sigma-1}}(2j)^{\tau(2j)^{\sigma-1}}\le A c_1^{j^{\sigma-1}}j^{\tau j^{\sigma-1}}
		\\&
		\Leftrightarrow 2^{\sigma-1}j^{\sigma-1}\log(c)+\tau2^{\sigma-1}j^{\sigma-1}\log(2j)\le\log(A)+j^{\sigma-1}\log(c_1)+\tau j^{\sigma-1}\log(j)
		\\&
		\Leftrightarrow 2^{\sigma-1}\left(\log(c)+\tau\log(2)\right)-\log(c_1)\le\frac{\log(A)}{j^{\sigma-1}}+\log(j)\tau\left(1-2^{\sigma-1}\right).
	\end{align*}
	As $j\rightarrow\infty$ the first summand on the right-hand side tends to $0$, whereas the second one tends to $-\infty$ and so does the right-hand side. This leads to a contradiction for any choice $c,c_1>0$.
	
	
	$(viii)$ For every $0<c_1<c_2$, we have that
	$$\lim_{j\to \infty} \left( \frac{M_j^{(c_2,\tau,\sigma)}}{M_j^{(c_1,\tau,\sigma)}} \right)^{1/j} = \lim_{j\to \infty} \left( \frac{c_2}{c_1} \right)^{j^{\sigma-1}}=\infty.$$

	\qed\enddemo
	
	\subsection{Results for PTT-classes}
	\label{sec:pttresults}
	By Theorem \ref{PPTasweightmatrixthm} we know that $\mathcal{E}_{[\tau,\sigma]}(U)$ can be identified with the matrix class $\mathcal{E}_{[\mathcal{M}^{\tau,\sigma}]}(U)$.
	From $(i)$ and $(iii)$ in Theorem \ref{fixedparammainresult} it follows immediately that $\mathcal{E}_{[\mathcal{M}^{\tau,\sigma}]}(U)$ contains the real analytic functions and the restrictions of entire functions and it is closed with respect to taking derivatives.\par
	By employing results from various works, let us now give a (non-exhaustive) list of results that hold for those matrix classes due to the regularity properties listed in Theorem \ref{fixedparammainresult}.\par

	\begin{itemize}
		\item[$(a)$] \textbf{Stability properties (\cite{characterizationstabilitypaper}):} $\mathcal{E}_{[\mathcal{M}^{\tau,\sigma}]}$ is...
		\begin{itemize}
			\item stable under composition,
			\item stable under solving ODEs,
			\item stable under inversion,
			\item inverse-closed.
		\end{itemize}
		This follows since by $(i)$, $(iii)$ and $(iv)$ the classes $\mathcal{E}_{[\mathcal{M}^{\tau,\sigma}]}$ satisfy all necessary properties such that Theorems 5 and 6 from \cite{characterizationstabilitypaper} are applicable.
		
		\item[$(b)$] \textbf{Almost analytic extensions (\cite{almostanalytic}):}
		$\mathcal{E}_{[\mathcal{M}^{\tau,\sigma}]}$-regularity of a function can be characterized by \emph{almost analytic extension}. This means that a function $f$ is in $\mathcal{E}_{[\mathcal{M}^{\tau,\sigma}]}(U)$ if and only if, for any \emph{quasiconvex} domain $V$ relatively compact in $U$, $f|_V$ can be extended to a function $F$ on $\mathbb{C}^d$ such that $\overline{\partial}F$ tends to $0$ sufficiently fast near $V$ (measured in terms of $\mathcal{M}^{\tau,\sigma}$).
		
		This follows since by $(i)$, $(ii)$ and $(iii)$ we have that $\widetilde{\mathcal{M}}^{\tau,\sigma}$ is a {\itshape regular} weight matrix in the sense of \cite[Def. 2.6]{almostanalytic}.
		
		\item[$(c)$] \textbf{Image of the Borel map (\cite{petzsche}, \cite{surjectivity}, \cite{maximal}, \cite{optimalRoumieu}, \cite{optimalBeurling}):} We have the following description of the image of the Borel map:
		\[
		j^\infty_0(\mathcal{E}_{[\mathcal{M}^{\tau,\sigma}]})=\Lambda_{[\mathcal{M}^{\tau,\sigma}]}.
		\]
		By $(vi)$ we have that each $M^{(c,\tau,\sigma)}$ has \hyperlink{gamma1}{$(\gamma_1)$} and the rest follows from the results of aforementioned papers.
		
		\item[$(d)$] \textbf{PTT-classes are not ``classical ultradifferentiable spaces'' (\cite{compositionpaper}, \cite{testfunctioncharacterization}):}
		By $(vii)$ and $(viii)$, neither $\mathcal{E}_{\{\mathcal{M}^{\tau,\sigma}\}}(U)$ nor $\mathcal{E}_{(\mathcal{M}^{\tau,\sigma})}(U)$ coincides (as vector spaces) with $\mathcal{E}_{\{M\}}(U)$, $\mathcal{E}_{\{\omega\}}(U)$, or, respectively, with $\mathcal{E}_{(M)}(U)$, $\mathcal{E}_{(\omega)}(U)$ for any weight sequence $M$ or any weight function $\omega$; see \cite[Prop. 4.6]{compositionpaper} and \cite[Cor. 3.17]{testfunctioncharacterization}.
		
		\item[$(e)$] \textbf{Nuclearity (\cite{testfunctioncharacterization}):} By $(iii)$ the classes $\mathcal{E}_{\{\mathcal{M}^{\tau,\sigma}\}}(U)$ and $\mathcal{E}_{(\mathcal{M}^{\tau,\sigma})}(U)$  are nuclear;
		see \cite[Prop. 7.2]{testfunctioncharacterization}.
		
		\item[$(f)$] \textbf{Almost harmonic extensions (\cite{debrouwvindas2020}):}
		$\mathcal{E}_{[\mathcal{M}^{\tau,\sigma}]}$-regularity of a function can be characterized by \emph{almost harmonic extension}. A little simplified, this means that a pair of functions $\phi_0, \phi_1$ is of  $\mathcal{E}_{[\mathcal{M}^{\tau,\sigma}]}$-regularity (on a subset of $\RR^d)$, if and only if (locally) there exists a function $\Phi$ (on a subset of $\RR^{d+1}$) such that the restriction of $\Phi$ to $\RR^d$ coincides with $\phi_0$, and the restriction of $\partial_y \Phi$ to $\RR^d$ coincides with $\phi_1$, and $\Delta \Phi$ tends to $0$ sufficiently fast near $\RR^d$ (measured in terms of $\mathcal{M}^{\tau,\sigma}$).\par

		This is due to the fact that Theorems \cite[Thm. 3.1, Thm. 3.2, Thm. 4.6]{debrouwvindas2020} can be applied to the classes $\mathcal{E}_{[\mathcal{M}^{\tau,\sigma}]}$:
		For this note that $[\mathfrak{M}.1]^{*}_{w}$ holds true by $(ii)$ in Theorem \ref{fixedparammainresult} (each strongly log-convex sequence even satisfies $(M.1)^{*}$), $[\mathfrak{M}.2]'$ by $(iii)$ and $(NA)$ is precisely \hyperlink{B-Comega}{$(\mathcal{M}_{(\on{C}^{\omega})})$} which is valid by $(i)$. Finally, also in \cite{debrouwvindas2020} a matrix $\mathcal{M}$ is called non-quasianalytic if each $M\in\mathcal{M}$ is non-quasianalytic and this is valid, in particular, by $(vi)$.
	\end{itemize}

	\section{PTT-limit classes as spaces defined by weight matrices} \label{PTTnonfixedparametersect}
	
	In \cite{PTT2015}, for fixed $\sigma\ge 1$,  the authors also consider limits with respect to the parameter $\tau$, i.e., the spaces $\mathcal{E}_{\infty,\sigma}$ and $\mathcal{E}_{0,\sigma}$, presented in  Subsection \ref{sec:ptt1},  which are endowed with the natural inductive resp. projective limit topology.

	
	The main reason to consider these classes is represented by the fact, that they are stable with respect to so-called \emph{ultradifferential operators.}  Observe here that a function $f$ lies in $\mathcal{E}_{\infty,\sigma}(U)$ if and only if there exists a uniform $\tau$ such that $f \in \mathcal{E}_{\{\tau,\sigma\}}(K)$ for all $K \subset \subset U$ and thus the limit classes $\mathcal{E}_{\infty,\sigma}(U)$ do not quite fit in the realm of Roumieu-type classes defined via weight matrices since quantifiers are exchanged. In the latter case it is allowed that $\tau$ is also depending on $K$; see \eqref{generalroumieu}.\vspace{6pt}
	
	In order to comment on this subtle difference, first for $\sigma\ge 1$ let us from now on consider the matrix (using the notation from \eqref{PPTconvenientmatrix})
	\begin{equation}\label{weightmatrixsigma}
		\mathcal{M}^\sigma:=\{M^{(\tau,\tau,\sigma)}: ~\tau >0\},\hspace{15pt}M^{(\tau,\tau,\sigma)}_j:=\tau^{j^{\sigma}}j^{\tau j^{\sigma}}.
	\end{equation}
	
	We get the following connection with the respective matrix class defined in terms of $\mathcal{M}^{\sigma}$.
	
	\begin{theorem}\label{PPTasweightmatrixthmtau}
		Let $U, V\subseteq \RR^d$ be open, and $\overline{V} \subset\subset U$. Then
		as locally convex vector spaces we get
		\begin{equation}\label{PPTasweightmatrixthmtauequ}
			\mathcal{E}_{\infty, \sigma}(U) \hookrightarrow \mathcal{E}_{\{ \mathcal{M}^\sigma \}}(U) \hookrightarrow \mathcal{E}_{\infty, \sigma}(V),\hspace{15pt}\mathcal{E}_{0,\sigma}(U)=\mathcal{E}_{(\mathcal{M}^{\sigma})}(U).
		\end{equation}
	\end{theorem}
	
	\demo{Proof}
	{\itshape The Roumieu case.}
	The first inclusion is clear from the definition of the respective spaces. For the second one, observe that any $f \in \mathcal{E}_{\{\mathcal{M}^\sigma\}}(U)$ lies in $\mathcal{E}_{\{M^{(\tau,\tau,\sigma)}\}}(\overline{V})$ for some $\tau$, i.e., there exists $h\ge 1$ and $A>0$ such that for all $x \in \overline{V}$ and $\alpha\in\NN^d$ we have
	\[
	|f^{(\alpha)}(x)| \le Ah^{|\alpha|} M^{(\tau,\tau,\sigma)}_{|\alpha|},
	\]
	and for any $\tau'>\tau$, we find $B>0$ such that for all $j$ we have
	\[
	h^{j} M^{(\tau,\tau,\sigma)}_{j}\le B M^{(\tau',\tau',\sigma)}_{j},
	\]
	which finishes the Roumieu case.
	
	\vspace{6pt}
	
	{\itshape The Beurling case.}
	Since there is a universal quantifier for the compact set, the index $\tau$, and the geometric factor $h$, we do not have to worry about interchanging those.
	The rest follows from Proposition \ref{PPTasweightmatrixproposition}.
	
	
	\qed\enddemo
	
	\begin{proposition}
		Let $\sigma\ge 1$ and $\mathcal{N}=\{N^{(x)}: x>0\}$ be a \hyperlink{Msc}{$(\mathcal{M}_{\on{sc}})$} and non-quasianalytic weight matrix. Suppose that we have
		\[
		\mathcal{E}_{\infty,\sigma}(U) \subseteq \mathcal{E}_{\{\mathcal{N}\}}(U).
		\]
		Then it follows that
		\[
		\mathcal{E}_{\infty,\sigma}(U) \subsetneq \mathcal{E}_{\{\mathcal{N}\}}(U).
		\]
	\end{proposition}
	
	\begin{proof}
		Similarly as in $(viii)$ in Theorem \ref{fixedparammainresult} we get
		\[
		\forall\;\tau_1>\tau>0:\;\;\;M^{(\tau,\tau,\sigma)} \lhd M^{(\tau_1,\tau_1,\sigma)}.
		\]
		The inclusion $\mathcal{E}_{\infty,\sigma}(U) \subseteq \mathcal{E}_{\{\mathcal{N}\}}(U)$ and the optimal function $\theta_{M^{(\tau_1,\tau_1,\sigma)}}$ (see \eqref{characterizingfct}) implies the following:
		\begin{equation}
			\label{eq:strongest}
			\forall\;\tau>0\;\exists\;x>0:\;\;\;M^{(\tau,\tau,\sigma)} \lhd N^{(x)}.
		\end{equation}
		Therefore, note that $\theta_{M^{(\tau_1,\tau_1,\sigma)}}\in\mathcal{E}_{\infty,\sigma}(U)$ for any $\tau_1>0$ since the estimate $|\theta_{M^{(\tau_1,\tau_1,\sigma)}}(t)|\le 2^{j+1}\tau_1^{j^{\sigma}}j^{\tau_1j^{\sigma}}$ holds globally on whole $\RR$; see again \cite[Thm. 1]{thilliez}, \cite[Lemma 2.9]{compositionpaper} and the detailed proof in \cite[Prop. 3.1.2]{diploma}.\vspace{6pt}
		
		Let $K_j$ be a sequence of mutually disjoint compact sets with non-empty interior contained in $U$ such that they accumulate at the boundary of $U$, i.e., for any compact set $K \subset \subset U$ there exists $j$ such that $K_j \cap K = \emptyset$. Let $S_j$ be also a sequence of compact sets such that $S_j \subseteq K_j^\circ$. Finally, let $x_j \in S_j$. Then by \cite[Cor. 3.12]{ultraextensionsurvey} there exists $\phi_j \in \mathcal{D}_{\{N^{(j)}\}}(K_j^\circ)$ such that $\phi_j \equiv 1$ on $S_j$. Let $\theta_j \in \mathcal{E}_{\{N^{(j)}\}}(\mathbb{R})$ be such that $|\theta_j^{(k)}((x_j)_1)|\ge N^{(j)}_k$ (where $(x_j)_1$ is the first component of $x_j$), and set $\Theta_j(x_1, \dots,x_d):= \theta_j(x_1)$. Finally, set
		\[
		h_j:= \Theta_j \phi_j, \quad h:= \sum_j h_j.
		\]
		Then clearly $h_j \in \mathcal{D}_{\{N^{(j)}\}}(K_j^\circ)$, and thus $h \in \mathcal{E}_{\{\mathcal{N}\}}(U)$. But $h \notin \mathcal{E}_{\{\tau,\sigma\}}$ for any $\tau $ (and therefore not in $\mathcal{E}_{\infty,\sigma}(U)$). To see this, take for given $M^{(\tau,\tau,\sigma)}$ some $j$ big enough to get \eqref{eq:strongest}. By taking $K=K_j$, one immediately gets $h \notin \mathcal{E}_{\{\tau,\sigma\}}(U)$.
	\end{proof}

	

	

	\begin{remark}\label{alternativedefrem}
		After a private communication, in the very recent paper \cite{TT2022} the authors already have taken into account this fact and included the definition of the limit classes
		\[
		\mathcal{E}^R_{\infty,\sigma}(U):= \bigcap_{K \subset\subset U} \bigcup_{\tau>0} \mathcal{E}_{\{\tau,\sigma\}}(K)=\bigcap_{K \subset\subset U} \bigcup_{\tau>0} \mathcal{E}_{\{\mathcal{M}^{\tau,\sigma}\}}(K)=\mathcal{E}_{\{\mathcal{M}^\sigma\}}(U),
		\]
		see \cite[$(2.12)$]{TT2022}. Note that this difference might be considered negligible in the light of Theorem \ref{PPTasweightmatrixthmtau} and for the Beurling-type both notions coincide; i.e.
		\[
		\mathcal{E}_{0,\sigma}(U)=\mathcal{E}_{(\mathcal{M}^\sigma)}(U).
		\]
		When considering the notion of germs of $\mathcal{E}_{\infty,\sigma}$-functions then also no difference occurs. For these equalities recall Theorem \ref{PPTasweightmatrixthmmatrix}.
	\end{remark}

	\subsection{Properties of the matrix $\mathcal{M}^{\sigma}$}
	
	From now on we focus on $\mathcal{E}_{[\mathcal{M}^\sigma]}(U)$ and gather several important growth and regularity properties for the crucial weight matrix $\mathcal{M}^{\sigma}$ from \eqref{weightmatrixsigma}.
	
	\begin{theorem}\label{matrixPsigmaproperties}
		Let $\sigma>1$, then the matrix $\mathcal{M}^{\sigma}$ has the following properties:
		\begin{itemize}
			\item[$(i)$] $\mathcal{M}^{\sigma}$ satisfies \hyperlink{B-Comega}{$(\mathcal{M}_{(\on{C}^{\omega})})$} and, more generally, even $\left(\frac{M^{(\tau, \tau,\sigma)}_j}{j!^{\alpha}}\right)^{1/j}\rightarrow+\infty$ for any $\alpha>0$ and any $\tau>0$.
			
			\item[$(ii)$] $\mathcal{M}^{\sigma}$ is equivalent to a matrix $\widetilde{\mathcal{M}}^{\sigma}$ all of whose sequences are strongly log-convex.
			
			\item[$(iii)$] $\mathcal{M}^{\sigma}$ satisfies both \hyperlink{R-mg}{$(\mathcal{M}_{\{\on{mg}\}})$} and \hyperlink{B-mg}{$(\mathcal{M}_{(\on{mg})})$}.
			
			\item[$(iv)$] $\mathcal{M}^{\sigma}$ has both \hyperlink{R-rai}{$(\mathcal{M}_{\{\on{rai}\}})$} and \hyperlink{B-rai}{$(\mathcal{M}_{(\on{rai})})$}; in fact we have that in both conditions we can choose the same index.
			
			\item[$(v)$] $\mathcal{M}^{\sigma}$ has both \hyperlink{R-FdB}{$(\mathcal{M}_{\{\on{FdB}\}})$} and \hyperlink{B-FdB}{$(\mathcal{M}_{(\on{FdB})})$}.
			
			\item[$(vi)$] $\mathcal{M}^{\sigma}$ has both \hyperlink{R-L}{$(\mathcal{M}_{\{\on{L}\}})$} and \hyperlink{B-L}{$(\mathcal{M}_{(\on{L})})$}.
			
			\item[$(vii)$] The sequences $M^{(c,c,\sigma)}$ are pair-wise not equivalent and $\mathcal{M}^{\sigma}$ has both $(\mathcal{M}_{\{\on{BR}\}})$ and $(\mathcal{M}_{(\on{BR})})$ in \cite[Sect. 4.1]{compositionpaper}.
			
			\item[$(viii)$] Each $M^{(\tau, \tau,\sigma)}$ is strongly non-quasianalytic; in fact we even have that $\gamma(M^{(\tau, \tau,\sigma)})=+\infty$ for all $\tau>0$.
		\end{itemize}
	\end{theorem}
	
	\demo{Proof}
	In order to shorten the notation, we write in this proof $M^{(c)}:= M^{(c,c,\sigma)}$ .\par
	
	$(i)$ This follows just as in the proof of Theorem \ref{fixedparammainresult}.
	
	
	
	$(ii)$ By the convention $0^0:=1$ we get $1=M^{(c)}_0$ and $c=M^{(c)}_1$ for all $c>0$. For each $c\ge 1$ the sequence $M^{(c)}$ is log-convex because $j\mapsto\log(j^{c j^{\sigma}})=cj^{\sigma}\log(j)$ is convex; more precisely one has $M^{(c)}\in\hyperlink{LCset}{\mathcal{LC}}$ for all $c\ge 1$.
	
	By replacing $\tau$ by $c$, we can repeat the arguments given in the proof of $(ii)$ in Theorem \ref{fixedparammainresult}. Since also the order is preserved we have that $\widetilde{\mathcal{M}}^{\sigma}$ is standard log-convex.
	
	$(iii)$ First, we test conditions \hyperlink{R-mg}{$(\mathcal{M}_{\{\on{mg}\}})$} and \hyperlink{B-mg}{$(\mathcal{M}_{(\on{mg})})$} on the diagonal, i.e., for $j=k\ge 1$. So let $A\ge 1$ and $c,c_1>0$ be given, then:
	\begin{align*}
		&(M^{(c)}_{2j})^{1/(2j)}\le A(M^{(c_1)}_j)^{1/j}\Leftrightarrow c^{(2j)^{\sigma-1}}(2j)^{c(2j)^{\sigma-1}}\le A c_1^{j^{\sigma-1}}j^{c_1j^{\sigma-1}}
		\\&
		\Leftrightarrow (2j)^{\sigma-1}\log(c)+c(2j)^{\sigma-1}\log(2j)\le\log(A)+j^{\sigma-1}\log(c_1)+c_1j^{\sigma-1}\log(j).
	\end{align*}
	We also have
	\begin{align*}
		c(2j)^{\sigma-1}\log(2j)\le j^{\sigma-1}c_1\log(j)&\Leftrightarrow 2^{\sigma-1}\log(2j)\le\frac{c_1}{c}\log(j)
		\\&
		\Leftrightarrow 2^{\sigma-1}\log(2)\le\log(j)(\frac{c_1}{c}-2^{\sigma-1}).
	\end{align*}
	In the Roumieu case, w.l.o.g. we take $c$ large enough to guarantee $c^{2^{\sigma-1}-1}\ge 1+2^{\sigma-1}$. Thus, when given such an index $c$ we choose $c_1:=c^{2^{\sigma-1}}>c$. Then, on the one hand clearly $(2j)^{\sigma-1}\log(c)=j^{\sigma-1}\log(c_1)$ and, on the other hand $\frac{c_1}{c}-2^{\sigma-1}\ge 1\Leftrightarrow c_1\ge(1+2^{\sigma-1})c$ because $c_1=c^{2^{\sigma-1}}\ge(1+2^{\sigma-1})c\Leftrightarrow c^{2^{\sigma-1}-1}\ge 1+2^{\sigma-1}$. Thus we are done when choosing $A$ sufficiently large and note that $c_1\rightarrow\infty$ as $c\rightarrow\infty$.
	
	In the Beurling case, w.l.o.g. we take given $c_1<1$ small enough to ensure $c_1<\frac{1}{1+2^{\sigma-1}}$ and then we set $c:=c_1^2<c_1$. Then, on the one hand, $c_1\ge c(1+2^{\sigma-1})$ is immediate and, second, we have (note that $\log(c_1)<0$)
	$$(2j)^{\sigma-1}\log(c)\le j^{\sigma-1}\log(c_1)\Leftrightarrow 2(2j)^{\sigma-1}\log(c_1)\le j^{\sigma-1}\log(c_1)\Leftrightarrow 2(2j)^{\sigma-1}\ge j^{\sigma-1},$$
	which holds for all $j\in\NN$.
	
	So far we have verified the desired properties on the diagonal (i.e., $j=k$). However, by the equivalence stated in $(ii)$ before also $\widetilde{\mathcal{M}}^{\sigma}$ has both \hyperlink{R-mg}{$(\mathcal{M}_{\{\on{mg}\}})$} and \hyperlink{B-mg}{$(\mathcal{M}_{(\on{mg})})$} verified on the diagonal. Thus \cite[Thm. 9.5.1, Thm. 9.5.3]{dissertation} applied to $\widetilde{\mathcal{M}}^{\sigma}$ yields the conclusion and by the equivalence we are done with $\mathcal{M}^{\sigma}$, too.\vspace{6pt}
	
	$(iv)$ \hyperlink{R-rai}{$(\mathcal{M}_{\{\on{rai}\}})$} follows by repeating the estimates from $(iv)$ in Theorem \ref{fixedparammainresult}. Note that we have for all $1\le c\le c_1$, $A\ge e$ and $1\le j\le k$
	\begin{align*}
		ec^{j^{\sigma-1}}j^{c j^{\sigma-1}-1}\le Ac_1^{k^{\sigma-1}}k^{c_1 k^{\sigma-1}-1}.
	\end{align*}
	
	Let $0<c_1\le c<1$, $A\ge 1$ and $1\le j\le k$, then
	\begin{align*}
		&ec_1^{j^{\sigma-1}}j^{c_1 j^{\sigma-1}-1}\le Ac^{k^{\sigma-1}}k^{c k^{\sigma-1}-1}
		\\&
		\Leftrightarrow j^{\sigma-1}\log(c_1)-k^{\sigma-1}\log(c)\le\log(A/e)+\left(c k^{\sigma-1}-1\right)\log(k)-\left(c_1 j^{\sigma-1}-1\right)\log(j).
	\end{align*}
	We take $c_1=c$ and repeat the computation from $(iv)$ in Theorem \ref{fixedparammainresult} when $\tau$ is replaced by $c$. This should be compared with $(ii)$ and recall that each strongly log-convex sequence satisfying $m_0=M_0=1$ has the property that $j\mapsto(m_j)^{1/j}$ is nondecreasing.\vspace{6pt}\vspace{6pt}
	
	$(v)$ This follows by $(i)$, $(iii)$ and $(iv)$; see \cite[Lemma 1 $(1)$]{characterizationstabilitypaper}.\vspace{6pt}
	
	$(vi)$ For all $h\ge 1$ (large) and all $0<c<c_1$ we can find some constant $A\ge 1$ (large) such that for all $j\in\NN_{>0}$:
	$$h^jM^{(c)}_j=h^jc^{j^{\sigma}}j^{cj^{\sigma}}\le Ac_1^{j^{\sigma}}j^{c_1j^{\sigma}}=AM^{(c_1)}_j\Leftrightarrow h\le A^{1/j}\left(\frac{c_1}{c}\right)^{j^{\sigma}}j^{j^{\sigma}(c_1-c)}.$$
	
	$(vii)$ The same estimate as given in $(vi)$ also yields the following property for $\mathcal{M}^{\sigma}$:
	$$\forall\;0<c<c_1:\;\;\;M^{(c)}\hyperlink{triangle}{\vartriangleleft}M^{(c_1)},$$
	hence both desired properties.
	
	$(viii)$ Follows analogously as in $(vi)$ in Theorem \ref{fixedparammainresult}.
	
	\qed\enddemo

	
	\subsection{PTT-limit classes as Braun-Meise-Taylor classes}
	\label{sec:function}
	\par
	Let $\sigma>1$ be given. Then, on the one hand $\mathcal{E}_{[\mathcal{M}^\sigma]}$ \textit{cannot} be described by a single weight sequence which follows by $(vii)$ in Theorem \ref{matrixPsigmaproperties}. On the other hand, the aim of this section is to show that it actually can be understood as a Braun-Meise-Taylor class. This question has very recently been studied and solved in \cite{TT2022} (for the modified defined limit classes mentioned in Remark \ref{alternativedefrem}). There the authors give precise asymptotics of $\omega$ in terms of the so-called \emph{Lambert function} $W$; cf. \cite[Prop. 3.1]{TT2022}. However, we give an independent proof of their main result by involving only weight matrix techniques.\vspace{6pt}
	
	Let us emphasize that for $\sigma=1$ this statement is not true: By \eqref{weightmatrixsigma} the matrix $\mathcal{M}^1:=\{M^{(\tau,\tau,1)}: ~\tau >0\}$ consists of sequences $M^{(\tau,\tau,1)}_j=\tau^jj^{\tau j}$ and hence $\mathcal{M}^1$ is equivalent to the Gevrey matrix $\mathcal{G}_0=\{(j!^{\tau})_{j\in\NN}: \tau>0\}$. From (the first paragraph in the proof of) \cite[Thm. 5.22]{compositionpaper} it follows that the corresponding weight matrix class cannot be described by a space given by a log-convex $M$ (in particular $M\in\hyperlink{LCset}{\mathcal{LC}}$) or by a weight function $\omega$.\vspace{6pt}
	
	We prove now an abstract result on the connection between weight sequences and their associated weight functions which is important in the ultradifferentiable setting on its own.
	
	\begin{lemma}\label{assofctrelationcharlemma}
		Let $M,N\in\hyperlink{LCset}{\mathcal{LC}}$. Then the following are equivalent:
		\begin{itemize}
			\item[$(i)$] $M$ and $N$ are related by
			\begin{equation}\label{assofctrelationcharlemmaequ}
				\exists\;c\in\NN_{>0}\;\exists\;A\ge 1\;\forall\;j\in\NN:\;\;\;N_j\le A(M_{cj})^{1/c}.
			\end{equation}
			\item[$(ii)$] The associated weight functions are related by
			$$\omega_M(t)=O(\omega_N(t)),\;\;\;t\rightarrow+\infty.$$
		\end{itemize}
		Moreover, the following are equivalent:
		\begin{itemize}
			\item[$(i)'$] $M$ and $N$ are related by
			\begin{equation}\label{assofctrelationcharlemmaequstrong}
				\forall\;c\in\NN_{>0}\;\exists\;A\ge 1\;\forall\;j\in\NN:\;\;\;(N_{cj})^{1/c}\le AM_j.
			\end{equation}
			\item[$(ii)'$] The associated weight functions are related by
			$$\omega_M(t)=o(\omega_N(t)),\;\;\;t\rightarrow+\infty.$$
		\end{itemize}
	\end{lemma}

	\demo{Proof}
	$(ii)\Rightarrow(i)$ We have $\omega_M(t)\le c\omega_N(t)+c$ for some $c\ge 1$ (large) and all $t\ge 0$. W.l.o.g. take $c\in\NN_{>0}$ and then \eqref{Prop32Komatsu} yields for all $j\in\NN$:
	\begin{align*}
		M_{cj}&=\sup_{t\ge 0}\frac{t^{cj}}{\exp(\omega_M(t))}\ge\frac{1}{e^c}\sup_{t\ge 0}\frac{t^{cj}}{\exp(c\omega_N(t))}=\frac{1}{e^c}\left(\sup_{t\ge 0}\frac{t^j}{\exp(\omega_N(t))}\right)^c=\frac{1}{e^c}N_j^c.
	\end{align*}
	Thus \eqref{assofctrelationcharlemmaequ} is shown with $A=e$.\vspace{6pt}
	
	$(i)\Rightarrow(ii)$ For given $M\in\RR_{>0}^{\NN}$ and $c\in\NN_{>0}$ we set
	\begin{equation*}\label{sequenceLC}
		\widetilde{M}^c_j:=(M_{cj})^{1/c},
	\end{equation*}
	hence $\widetilde{M}^1\equiv M$ is clear. If $M$ is log-convex, then each $\widetilde{M}^c$ as well and $\widetilde{M}^c_0=1$ if $M_0=1$. If $M\in\hyperlink{LCset}{\mathcal{LC}}$, then $\widetilde{M}^c\in\hyperlink{LCset}{\mathcal{LC}}$ (for some/any $c\in\NN_{>0}$). Thus by definition and assumption we get
	\begin{equation}\label{auxiliaryequ}
		\exists\;A\ge 1\;\forall\;t\ge 0:\;\;\;\omega_{\widetilde{M}^c}(t)\le\omega_N(t)+\log(A).
	\end{equation}
	We obtain for all $j\in\NN$ and $c\in\NN_{>0}$
	$$\widetilde{M}^c_j:=(M_{cj})^{1/c}=\left(\sup_{t\ge 0}\frac{t^{cj}}{\exp(\omega_M(t))}\right)^{1/c}.$$
	Moreover, since $M\in\hyperlink{LCset}{\mathcal{LC}}$ we get $\omega_M(t)=0$ for $t\in[0,1]$ (i.e., normalization) and so:
	\begin{align*}
		\widetilde{M}^c_j&=\sup_{t\ge 0}\frac{t^j}{\exp(c^{-1}\omega_M(t))}=\sup_{t\ge 1}\frac{t^j}{\exp(c^{-1}\omega_M(t))}=\exp(\sup_{t\ge 1}j\log(t)-c^{-1}\omega_M(t))
		\\&
		=\exp(\sup_{s\ge 0}j s-c^{-1}\omega_M(e^s))=:\exp(\varphi^{*}_{c^{-1}\omega_M}(j)).
	\end{align*}

	Thus we may apply \cite[Lemma 2.5]{sectorialextensions} to $\omega\equiv c^{-1}\omega_M$ and get
	\[
	c^{-1}\omega_M\hyperlink{sim}{\sim}\omega_{\widetilde{M}^c}.
	\]
	
	More precisely, by setting the weight matrix parameter $x=1$, we see
	\begin{equation}\label{goodequivalenceclassic}
		\exists\;D>0\;\forall\;t\ge 0:\;\;\;\omega_{\widetilde{M}^c}(t)\le c^{-1}\omega_M(t)\le 2\omega_{\widetilde{M}^c}(t)+D.
	\end{equation}
	Combining \eqref{goodequivalenceclassic} with \eqref{auxiliaryequ} immediately yields
	$$\frac{1}{2c}\omega_M(t)-\frac{D}{2}\le\omega_{\widetilde{M}^c}(t)\le\omega_N(t)+\log(A)\Longrightarrow\omega_M(t)\le 2c\omega_N(t)+Dc+2c\log(A).$$
	Thus $\omega_M(t)=O(\omega_N(t))$ as $t\rightarrow+\infty$ is verified.\vspace{6pt}
	
	$(ii)'\Rightarrow(i)'$ For all $c\in\NN_{>0}$ we can find $D\ge 1$ such that $\omega_M(t)\le\frac{1}{c}\omega_N(t)+D$ for all $t\ge 0$ and so, analogously as before, we obtain \eqref{assofctrelationcharlemmaequstrong} with $A:=e^D$.
	
	$(i)'\Rightarrow(ii)'$ Using the notation from above, \eqref{auxiliaryequ} transfers into
	\begin{equation}\label{auxiliaryequstrong}
		\forall\;c\in\NN_{>0}\;\exists\;A\ge 1\;\forall\;t\ge 0:\;\;\;\omega_M(t)\le\omega_{\widetilde{N}^c}(t)+\log(A).
	\end{equation}
	Then we follow the arguments in $(i)\Rightarrow(ii)$ and combine \eqref{auxiliaryequstrong} with the first half from \eqref{goodequivalenceclassic} applied to $N$ in order to get
	$$\forall\;c\in\NN_{>0}\;\exists\;A\ge 1\;\forall\;t\ge 0:\;\;\;\omega_M(t)\le c^{-1}\omega_N(t)+\log(A).$$
	Thus $\omega_M(t)=o(\omega_N(t))$ is verified.
	\qed\enddemo
	
	The importance of Lemma \ref{assofctrelationcharlemma} is that it enables the possibility to express all requirements in \cite[Cor. 3.17 $(ii)$]{testfunctioncharacterization} purely in terms of the given matrix $\mathcal{N}$ directly:
	
	\begin{corollary}\label{assofctrelationcharlemmacor}
		Let $\mathcal{M}$ be \hyperlink{Msc}{$(\mathcal{M}_{\on{sc}})$}. Then as locally convex vector spaces
		$$\mathcal{E}_{[\mathcal{M}]}=\mathcal{E}_{[\omega]},$$
		with $\omega$ being a weight function in the sense of Braun-Meise-Taylor (see \cite{BraunMeiseTaylor90}, \cite{testfunctioncharacterization}) if and only if there exists a \hyperlink{Msc}{$(\mathcal{M}_{\on{sc}})$} matrix $\mathcal{N}=\{N^{(\alpha)}: \alpha>0\}$ which is $R$- resp. $B$-equivalent to $\mathcal{M}$ and such that
		\begin{itemize}
			\item[$(*)$] $\mathcal{N}$ has $(\mathcal{M}_{[\on{L}]})$,
			
			\item[$(*)$] $\mathcal{N}$ has $(\mathcal{M}_{[\on{mg}]})$,
			
			\item[$(*)$] $\mathcal{N}$ has (cf. \eqref{assofctrelationcharlemmaequ})
			\begin{equation}\label{assofctrelationcharlemmacorequ}
				\forall\;\alpha,\beta>0\;\exists\;c\in\NN_{>0}\;\exists\;A\ge 1\;\forall\;j\in\NN:\;\;\;N^{(\alpha)}_j\le A(N^{(\beta)}_{cj})^{1/c}.
			\end{equation}
		\end{itemize}
	\end{corollary}
	
	\begin{remark}
		Any $\omega_{N^{(\alpha)}}$ (for $N^{(\alpha)}\in \mathcal{N}$) is a valid choice for $\omega$ in Corollary \ref{assofctrelationcharlemmacor}.
		
		Note that \eqref{assofctrelationcharlemmacorequ} is clearly preserved under R- and B-equivalence of weight matrices.
	\end{remark}
	
	In particular we can apply this statement to PTT-limit classes and get the following.
	
	\begin{theorem}\label{varyingtaucharact}
		Let $\sigma>1$ and put $\omega^{(\sigma)} := \omega_{M^{(1,1,\sigma)}}$; i.e. the associated weight function of $M^{(1,1,\sigma)}=(j^{j^{\sigma}})_{j\in\NN}$. Then as locally convex vector spaces we get
		\begin{equation}\label{varyingtaucharactequ}
			\mathcal{E}_{[\mathcal{M}^\sigma]}=\mathcal{E}_{[\omega^{(\sigma)}]}.
		\end{equation}
		Moreover, all associated weight functions of the matrix $\mathcal{M}^\sigma$ are equivalent, i.e.,
		$$\forall\;h,h', \tau, \tau'>0:\;\;\;\omega_{M^{(h,\tau,\sigma)}}\hyperlink{sim}{\sim}\omega_{M^{(h',\tau',\sigma)}},$$
		and consequently in \eqref{varyingtaucharactequ} we can replace $\omega^{(\sigma)}$ by any $\omega_{M^{(h,\tau,\sigma)}}$.
	\end{theorem}
	
	\begin{proof}
		We only have to verify \eqref{assofctrelationcharlemmacorequ} in Corollary \ref{assofctrelationcharlemmacor} for the matrix $\mathcal{M}^\sigma$ (the first two conditions are contained in Theorem \ref{matrixPsigmaproperties}). But this is clear for our concrete matrix $\mathcal{M}^\sigma$  since \eqref{assofctrelationcharlemmacorequ} holds for $\beta\geq \alpha$ taking $c=A=1$ and for $\beta<\alpha$ we have that
		$$j\mapsto j^\sigma \big(\log(\alpha) - c^{\sigma-1} \log (\beta) + (\alpha-\beta c^{\sigma-1}) \log (j) - (\beta c^{\sigma-1})\log(c)   \big)$$
		is bounded from above for $c$ large enough.	\end{proof}
	
	\begin{remark}
		The previous result fails for $\sigma=1$: The Gevrey matrix $\mathcal{G}_0(=\mathcal{M}^1)$ clearly satisfies $(\mathcal{M}_{[\on{L}]})$ and $(\mathcal{M}_{[\on{mg}]})$ but \eqref{assofctrelationcharlemmacorequ} is violated when taking e.g. $\alpha=2\beta$.
	\end{remark}


	\subsection{Results for PTT-limit classes}
	\label{sec:pttlimresults}
	As seen in the previous section the PTT-limit classes can be represented by certain Braun-Meise-Taylor classes (defined by the weight function $\omega^{(\sigma)}:= \omega_{M^{(1,1,\sigma)}}$).
	Let us now give additional properties available for PTT-limit classes that follow from this representation (and the properties listed in Theorem \ref{matrixPsigmaproperties}).
	

	\begin{itemize}

		\item[$(a)$] \textbf{Stability properties (\cite{characterizationstabilitypaper}):}
		$\mathcal{E}_{[\mathcal{M}^\sigma]}$ is...
		\begin{itemize}
			\item stable under composition,
			\item stable under solving ODEs,
			\item stable under inversion,
			\item inverse-closed.
		\end{itemize}

		This follows since by $(i)$, $(iii)$ and $(iv)$ in Theorem \ref{matrixPsigmaproperties} the classes $\mathcal{E}_{[\mathcal{M}^{\sigma}]}$ satisfy all necessary properties such that Theorems 5 and 6 from \cite{characterizationstabilitypaper} are applicable.
		
		\item[$(b)$] \textbf{Almost analytic extensions (\cite{almostanalytic}):} $\mathcal{E}_{[\mathcal{M}^\sigma]}$-regularity of a function can be characterized by \emph{almost analytic extension}. This means mutatis mutandis the same as already outlined in Section \ref{sec:pttresults}.

In addition {\itshape both} \cite[$(7.1), (7.2)$]{almostanalytic} hold true; $(7.1)$ is precisely \hyperlink{R-mg}{$(\mathcal{M}_{\{\on{mg}\}})$} and $(7.2)$ is \hyperlink{R-FdB}{$(\mathcal{M}_{\{\on{FdB}\}})$}. In particular one can deduce (among other things) an ultradifferentiable elliptic regularity theorem.
		
		\item[$(c)$] \textbf{The image of the Borel map (\cite{petzsche}, \cite{surjectivity}, \cite{maximal}, \cite{optimalRoumieu}, \cite{optimalBeurling}):}
		We have the following description of the image of the Borel map:
		\[
		j^\infty_0(\mathcal{E}_{[\mathcal{M}^\sigma]})=\Lambda_{[\mathcal{M}^\sigma]}.
		\]
		By $(viii)$ we have that each $M^{(\tau,\tau,\sigma)}$ has \hyperlink{gamma1}{$(\gamma_1)$} and the rest follows from the results of the aforementioned papers.
		
		\item[$(d)$] \textbf{Cartesian closedness (\cite{convenientweightmatrix}):}
		For $E_1,E_2,F$ \emph{convenient} vector spaces and $U_i \subseteq E_i$ $c^\infty$-open (for the definitions consult \cite{convenientweightmatrix}, or the thorough treatment \cite{KM97}) one has as convenient vector spaces
		\[
		\mathcal{E}_{[\mathcal{M}^\sigma]}(U_1 \times U_2, F) \cong \mathcal{E}_{[\mathcal{M}^\sigma]}(U_1, \mathcal{E}_{[\mathcal{M}^\sigma]}(U_2, F)).
		\]
		
		This follow since by $(i)$, $(ii)$, $(iii)$ and $(iv)$ the classes $\mathcal{E}_{[\mathcal{M}^{\sigma}]}$ form cartesian closed categories  due to \cite[Thm. 5.9, 6.2]{convenientweightmatrix}.
		
		\item[$(e)$] \textbf{A result on powers (\cite{joris1}, \cite{joris2}):}
		If for two integers $p,q$ with $\gcd(p,q)=1$ and some function $f$, we have $f^p,f^q \in \mathcal{E}_{[\mathcal{M}^\sigma]}$, then we already have $f \in \mathcal{E}_{[\mathcal{M}^\sigma]}$.\par
		This follows since by $(iii)$ we have that $\mathcal{M}^\sigma$ satisfies $(\mathcal{M}_{[\on{mg}]})$, and, by $(ii)$, $\widetilde{\mathcal{M}}^{\sigma}$ has the desired properties. Thus \cite[Thm. 1.1, Thm. 4.1]{joris2} (cf. also \cite[Thm. 4.2]{joris1}) is applicable and immediately gives the claim.
		

		\item[$(f)$] \textbf{Almost harmonic extensions (\cite{debrouwvindas2020}):} $\mathcal{E}_{[\mathcal{M}^\sigma]}$-regularity can be characterized by \emph{almost harmonic extension}. The rest is mutatis mutandis the same as outlined in Section \ref{sec:pttresults}.
		
		\item[$(g)$] \textbf{Nuclearity (\cite{testfunctioncharacterization}):} By $(iii)$ the classes $\mathcal{E}_{\{\mathcal{M}^{\sigma}\}}(U)$ and $\mathcal{E}_{(\mathcal{M}^{\sigma})}(U)$  are nuclear;
		see \cite[Prop. 7.2]{testfunctioncharacterization}.
	\end{itemize}

	\subsection{A further possible result}
	\label{sec:possibleresult}
	In \cite{debrouwkalmes2022}, the authors consider ultradistributional boundary values of constant coefficient hypoelliptic partial differential operators. \emph{Ultradistributional} is understood in the framework of Denjoy-Carleman classes, i.e., classes defined via weight sequences.
	They require, apart from the normalization condition $1=M_0=M_1$ as basic assumptions for $M$ log-convexity, \hyperlink{mg}{$(\on{mg})$}, \hyperlink{nq}{$(\on{nq})$} and \hyperlink{beta3}{$(\beta_3)$}; see \cite[Def. 2.6]{debrouwkalmes2022} (there \hyperlink{beta3}{$(\beta_3)$} is denoted by $(M.2)^{*}$).
	
	It seems to be reasonable that the results from \cite{debrouwkalmes2022} can be transferred to $\mathcal{E}_{[\mathcal{M}^\sigma]}$: Note that each sequence satisfies all standard assumptions except \hyperlink{mg}{$(\on{mg})$} because by $(viii)$ in Theorem \ref{matrixPsigmaproperties} and \cite[Thm. 3.11]{index} we get \hyperlink{beta3}{$(\beta_3)$} and even condition $(M.4)_a$ for any $a>0$; see \cite[Def. 2.7]{debrouwkalmes2022} resp. \cite[Thm. 3.11 $(ii)$]{index}. (Also $M^{(\tau,\tau,\sigma)}_1=1$ is violated for $\tau\neq 1$ but which can be achieved by switching to an equivalent weight matrix.)
	
	Then one can try to compensate the failure of \hyperlink{mg}{$(\on{mg})$} by applying \hyperlink{R-mg}{$(\mathcal{M}_{\{\on{mg}\}})$} resp. \hyperlink{B-mg}{$(\mathcal{M}_{(\on{mg})})$} instead and which is valid by $(iii)$ in Theorem \ref{matrixPsigmaproperties}.\vspace{6pt}
	
	However, by $(vii)$ in Theorem \ref{fixedparammainresult} both generalized moderate-growth-type conditions fail for $\mathcal{M}^{\tau,\sigma}$ and so a generalization of the proofs from \cite{debrouwkalmes2022} to the PTT-classes $\mathcal{E}_{[\tau,\sigma]}=\mathcal{E}_{[\mathcal{M}^{\tau,\sigma}]}$ is not clear. (The other standard properties, except $M^{(c,\tau,\sigma)}_1=1$ for each $c>0$, hold true for each sequence in $\mathcal{M}^{\tau,\sigma}$.)

	\bibliographystyle{plain}

\end{document}